\documentclass[12pt,reqno] {amsart}
\usepackage{xcolor}
\usepackage{eucal}
\usepackage{amsthm}
\usepackage{amsfonts}
\usepackage{amsmath}
\usepackage{amssymb}
\usepackage{comment}
\usepackage{epsfig}
\usepackage[pagebackref,hypertexnames=false, colorlinks, citecolor=red,linkcolor=blue, urlcolor=red]{hyperref}
\newcommand{\ncom}{\newcommand}

\ncom{\ul}{\underline}
\ncom{\ol}{\overline}
\ncom{\bq}{\begin{equation}}
\ncom{\eq}{\end{equation}}
\ncom{\beqn}{\begin{eqnarray*}}
\ncom{\eeqn}{\end{eqnarray*}}
\ncom{\beq}{\begin{eqnarray}}
\ncom{\eeq}{\end{eqnarray}}
\ncom{\nno}{\nonumber}
\ncom{\rar}{\rightarrow}
\ncom{\Rar}{\Rightarrow}
\ncom{\noin}{\noindent}
\ncom{\bc}{\begin{centre}}
\ncom{\ec}{\end{centre}}
\ncom{\sz}{\scriptsize}
\ncom{\rf}{\ref}
\ncom{\sgm}{\sigma}
\ncom{\Sgm}{\Sigma}
\ncom{\dt}{\delta}
\ncom{\Dt}{Delta}
\ncom{\s}{\underline{s}}
\ncom{\p}{\underline{p}}
\ncom{\lmd}{\lambda}
\ncom{\Lmd}{\Lambda}
\ncom{\eps}{\epsilon}
\ncom{\pcc}{\stackrel{P}{>}}
\ncom{\dist}{{\rm\,dist}}
\ncom{\sspan}{{\rm\,span}}
\ncom{\re}{{\rm Re\,}}
\ncom{\im}{{\rm Im\,}}
\ncom{\sgn}{{\rm sgn\,}}
\ncom{\ba}{\begin{array}}
\ncom{\ea}{\end{array}}
\ncom{\eop}{\hfill{{\rule{2.5mm}{2.5mm}}}}
\ncom{\eoe}{\hfill{{\rule{1.5mm}{1.5mm}}}}
\ncom{\eof}{\hfill{{\rule{1.5mm}{1.5mm}}}}
\ncom{\hone}{\mbox{\hspace{1em}}}
\ncom{\htwo}{\mbox{\hspace{2em}}}
\ncom{\hthree}{\mbox{\hspace{3em}}}
\ncom{\hfour}{\mbox{\hspace{4em}}}
\ncom{\hsev}{\mbox{\hspace{7em}}}
\ncom{\vone}{\vskip 2ex}
\ncom{\cH}{{\mathcal H}}
\ncom{\vtwo}{\vskip 4ex}
\ncom{\vonee}{\vskip 1.5ex}
\ncom{\vthree}{\vskip 6ex}
\ncom{\vfour}{\vspace*{8ex}}
\ncom{\norm}{\|\;\;\|}
\ncom{\integ}[4]{\int_{#1}^{#2}\,{#3}\,d{#4}}
\ncom{\inp}[2]{\langle{#1},\,{#2} \rangle}
\ncom{\Inp}[2]{\big\langle{#1},\,{#2} \big\rangle}
\ncom{\vspan}[1]{{{\rm\,span}\#1 \}}}
\ncom{\dm}[1]{\displaystyle {#1}}
\ncom{\Hom}{\operatorname{Hom}}
\ncom{\Hol}{\operatorname{Hol}}
\ncom{\Ps}{\mathcal P_{\underline{s}}}
\ncom{\hl}{\mathcal H}

\ncom{\defin} {\overset {\text {\rm def} }{=}}

\newtheorem{theorem}{\bf Theorem}[section]
\newtheorem{proposition}[theorem]{\bf Proposition}%[section]
\newtheorem{corollary}[theorem]{\bf Corollary}%[section]
\newtheorem{lemma}[theorem]{\bf Lemma}%[section]
\newtheorem{remark}[theorem]{\bf Remark}%[section]
\newtheorem{definition}[theorem]{\bf Definition}%[section]

\def \s{\underline{s}}

\renewcommand{\epsilon}{\varepsilon}
\renewcommand{\kappa}{\varkappa}

\begin{document}

\title[CNP]{Complete Nevanlinna-Pick property of $\mathbb K$-Invariant Reproducing Kernels}

\author[M. Engli\v{s}]{Miroslav Engli\v{s}}
\address[M. Engli\v{s}]{Mathematical Institute in Opava, Silesian University in Opava, Na Rybn\'icku 626/1, 74601, Opava, Czech Republic and Institute of Mathematics, Czech Academy of Sciences, \v{Z}itn\'{a}~25, 11567 Prague 1, Czech Republic} 
\email{englis@math.cas.cz}
\author[S. Hazra ]{Somnath Hazra}
\address[S. Hazra]{ Institute for Advancing Intelligence, TCG Centres for Research and Education in Science and Technology, Kolkata, India} 
\email{somnath.hazra@tcgcrest.org}
\author[P. Pramanick]{Paramita Pramanick}
\address[P. Pramanick]{Statistics and Mathematics Unit, Indian Statistical Institute, Kolkata 700108, India}
\email{paramitapramanick@gmail.com}

\thanks{The first-named author was supported by GA\v{C}R grant no. 25-18042S and RVO funding for I\v{C}O 67985840. Support for the work of the third named author was provided by the Department of Science and Technology (DST) in the form of the Inspire Faculty Fellowship (Ref No. DST/INSPIRE/04/2023/001530).}


\date{}

\begin{abstract}
Let $\Omega$ be a Cartan domain and $K = \sum_{\underline s}a_{\underline s}K_{\underline s}$ be a $\mathbb K$-invariant kernel on $\Omega$. In this article, we first obtain a necessary  condition on $K$ to have the complete Nevanlinna-Pick property in terms of the sequence $\{a_{\underline s}\}_{\underline s}$ with the assumption that each $a_{\underline s}$ is non-zero and $K$ is non-vanishing. This generalizes the well-known Kaluza's Lemma in the context of $\mathbb K$-invariant kernels.

The notion of the characteristic function of the classical Sz.-Nagy--Foias Theory is extended to a commuting tuple of $\frac{1}{K}$-contraction where $K$ is an irreducible $\mathbb K$-invariant kernel. An explicit construction of the characteristic function of a $\frac{1}{K}$-contraction is provided. A characterization of a $\mathbb K$-invariant kernel with the complete Nevanlinna-Pick property is obtained via the existence of characteristic functions associated with $\frac{1}{K}$-contractions. 
\end{abstract}

\renewcommand{\thefootnote}{}

\maketitle
%%%%%%%%%%%%%%%%%%%%%%%%%%%%%%%%%%%%%%%%%%%%%%%%%%%%%%%%%%%%%%%%%%%%%%%%%%%%%%%%%%%%%%%%%%%
%%%%%%%%%%%%%%%%%%%%%%%%%%%%%%%%%%%%%%%%%%%%%%%%%%%%%%%%%%%%%%%%%%%%%%%%%%%%%%%%%%%%%%%%%%%

\section{Introduction}
Let $\Omega$ be a Cartan domain in ${\mathbb C}^d$ of rank $r$. These domains are the natural generalization of open unit disc in one complex
variable and open Euclidean unit ball in several complex variables. A complete classification of irreducible Cartan domains is given by \'E. Cartan \cite{Ca}.
The numerical invariants $(r, a, b)$ determine the domain $\Omega $ uniquely up to biholomorphic equivalence. The dimension $d$ is related to the numerical invariants $(r,a,b)$  by the relation $\frac{d}{r}=1+\frac{a}{2}(r-1)+b.$ For more details on Cartan domains, we refer to \cite{Loos},\cite{Arazy}.
Let $G$ be the connected component of identity in $\rm{Aut}(\Omega)$, the group of all the biholomorphic automorphisms of $\Omega.$ Let $\mathbb{K}=\{g \in G : g(0)=0\}$ be the maximal compact subgroup of $G$. Every irreducible Cartan domain $\Omega$ of rank $r$ can be realized as an open unit ball of
a Cartan factor $Z={\mathbb C}^d.$
The space of analytic polynomials $\mathcal P(Z)$ on $Z$ has a natural action of the group $\mathbb K$ by composition, that is, $(k\cdot p)(z)=p(k^{-1} \cdot z)$, $k\in \mathbb K,~ p\in \mathcal P(Z).$

An $r$-tuple of non-negative integers $\s=(s_1, \ldots , s_r)$ is called a {\it signature} if $s_1 \geq \ldots \geq s_r\geq 0$. The set of all signatures is denoted by $\vec{\mathbb N}_0^r$ and $\underline 0$ represents the signature $(0, \ldots, 0)$. If $\underline p = (p_1, \ldots, p_r)$ and $\underline q = (q_1, \ldots, q_r)$ are two signatures, then the notation $\underline p \geq \underline q$ means that $p_i \geq q_i$ for each $1 \leq i \leq r$ and the notation $\underline p > \underline q$ means that that there exists $1 \leq j \leq r$ such that $p_j > q_j$. Let $\{\epsilon_1, \ldots, \epsilon_r\}$ denote the standard ordered basis of $\mathbb C^r$. Note that among all elements of the standard ordered basis of $\mathbb C^r$, only $\epsilon_1$ is a signature. The action of $\mathbb K$ on $\mathcal{P}(Z)$ is not irreducible. In fact, $\mathcal P(Z)$ decomposes into irreducible, mutually $\mathbb{K}$ - inequivalent subspaces $\mathcal P_{\s}$, where $\s$ is a signature. Such a decomposition of $\mathcal{P}(Z)$ is called the {\it Peter-Weyl decomposition} in \cite[Section 3]{Hu} (see also \cite[page 21]{Arazy}). The space $\mathcal{P}_{\s }$ with respect to the Fischer-Fock inner product 
\[\langle p,q\rangle_{F}:=\frac{1}{\pi^d}\int_{\mathbb{C}^d}p(z)\overline{q(z)}e^{-| z|^2}dm(z)\]
is a reproducing kernel Hilbert space, where $dm(z)$ is the Lebesgue measure. The reproducing kernel of $\mathcal{P}_{\s}$ is denoted by $K_{\s}$. 
Let $d_{\s}$ be the dimension of $\mathcal{P}_{\s}$ for every $\s \in \vec{\mathbb N}_0^r.$ Throughout this article, we fix an orthonormal basis $\{\psi_\alpha^{\s}(z)\}_{\alpha=1}^{d_{\s}}$ of $\mathcal{P}_{\s}$ with respect to the Fischer-Fock inner product. Then, the reproducing kernel $K_{\s}$ is given by
\beq \label{ortho} K_{\s}(\boldsymbol{z}, \boldsymbol{w})= \sum_{\alpha=1}^{d_{\s}} \psi_\alpha^{\s}(\boldsymbol{z}) \overline{\psi_\alpha^{\s}(\boldsymbol{w})}. \eeq

A non-negative definite kernel $K : \Omega \times \Omega \to \mathbb C$ is said to be a $\mathbb K$-invariant kernel if $K(k \cdot \boldsymbol{z}, k \cdot \boldsymbol{w}) = K(\boldsymbol{z}, \boldsymbol{w})$ holds for every $\boldsymbol{z}, \boldsymbol{w} \in \Omega$ and $k \in \mathbb K$. Note that, for each signature $\s$, the reproducing kernel $K_{\s}$ of $\mathcal{P}_{\s}$ is a $\mathbb K$-invariant kernel. If $K : \Omega \times \Omega \to \mathbb C$ is a $\mathbb K$-invariant kernel which is holomorphic in $\boldsymbol{z}$ and $\bar{\boldsymbol{w}}$, then there exists a sequence of non-negative real numbers $\{a_{\s}\}_{\s \in \vec{\mathbb N}_0^r}$ such that
$$K(\boldsymbol{z}, \boldsymbol{w}) = \displaystyle \sum_{\s \in \vec{\mathbb N}_0^r} a_{\s} K_{\s}(\boldsymbol{z}, \boldsymbol{w}),$$
for every $\boldsymbol{z}, \boldsymbol{w} \in \Omega$ (cf. \cite{Arazy}). In this article, we consider non-negative definite kernels on $\Omega \times \Omega$ that are holomorphic in $\boldsymbol{z}$ and $\bar{\boldsymbol{w}}$. In particular, if $\Omega$ is an Euclidean unit ball $\mathbb B_d$ in $\mathbb C^d$, and $\mathbb K$ is the group of all $d \times d$ unitary matrices $\mathcal{U}(d)$, then 
$$K(\boldsymbol{z}, \boldsymbol{w}) = \displaystyle \sum_{n = 0}^{\infty} a_{n} \langle\boldsymbol{z}, \boldsymbol{w}\rangle^n,\,\,\boldsymbol{z}, \boldsymbol{w} \in \mathbb B_d,$$
where $\{a_n\}_{n \geq 0}$ is a sequence of non-negative real numbers. The $\mathcal U(d)$-invariant and $\mathbb K$-invariant kernels have been studied extensively over the past few decades (\cite{Arazy, Tiret, Hartz, GKP}). 

A non-negative definite kernel $K : \Omega \times \Omega \to \mathbb C$ is said to be normalized at a point $z_0$ in $\Omega$, if $K(z, z_0) = 1$ for every $z \in \Omega$. Note that if  $K(\boldsymbol{z}, \boldsymbol{w}) = \sum_{\s \in \vec{\mathbb N}_0^r} a_{\s} K_{\s}(\boldsymbol{z}, \boldsymbol{w}),\,\,\boldsymbol{z}, \boldsymbol{w} \in \Omega,$
is a $\mathbb K$-invariant kernel, then $K(\boldsymbol{z}, \boldsymbol{0}) = a_{\underline 0}$ for every $\boldsymbol{z} \in \Omega$. As a consequence, $\frac{1}{a_{\underline 0}} K$ is a $\mathbb K$-invariant kernel which is normalized at $\boldsymbol{0} \in \Omega$. \textit{In this article, we always work with a normalized $\mathbb K$-invariant kernel, that is, a kernel $K = \sum_{\s \in \vec{\mathbb N}_0^r} a_{\s} K_{\s}$ on $\Omega$ with $a_{\underline 0} = 1$. We also assume that $a_{\s} > 0$ for every signature $\s$.} 

Given a non-negative definite kernel $K: \Omega \times \Omega \to \mathbb C$, the Moore-Aronszajn Theorem says that there exists a Hilbert space $H_K$ such that $K$ is the reproducing kernel of $H_K$. \textit{Throughout this article, $H_K$ represents the Hilbert space with the reproducing kernel $K$.} Our objective is to classify $\mathbb K$-invariant kernels that have the complete Nevanlinna-Pick property. Let us recall the definition of the complete Nevanlinna-Pick property of a reproducing kernel.

\begin{definition}
Let $X$ be a subset of $\mathbb C^d$. A reproducing kernel $K : X \times X \to \mathbb C$ is said to have $M_{p\times p}$ Nevanlinna-Pick property if, whenever $\boldsymbol z_1,\ldots, \boldsymbol z_N\in X$ and  $W_1, \ldots, W_N$ are $p \times p$ matrices such that 
$$\left(\!\left((I-W_iW_j^*)K(\boldsymbol z_i, \boldsymbol z_j)\right)\!\right)_{i, j = 1}^N,$$
is non-negative definite, then there exists a multiplier $\phi$ in the closed unit ball of the multiplierr algebra $\text{Mult}( H_K \otimes \mathbb C^p, H_K \otimes \mathbb C^p)$ such that $\phi(\boldsymbol z_i)=W_i, i=1, \ldots, N.$ The kernel $K$ is said to have the complete Nevanlinna-Pick property if it has the $M_{p\times p}$ Nevanlinna-Pick property for all positive integers $p.$ A reproducing kernel with the complete Nevanlinna-Pick property is called the complete Nevanlinna-Pick (CNP) kernel.
\end{definition}

Let $K_{\mathcal{U}}(\boldsymbol{z}, \boldsymbol{w}) = \sum_{n \geq 0} a_n \langle \boldsymbol{z}, \boldsymbol{w}\rangle^n$ be a non-vanishing $\mathcal{U}(d)$-invariant kernel on $\mathbb B_d$ with each $a_n$ being positive. Then, there exists a sequence of real numbers $\{\hat{b}_n\}_{n \geq 0}$ such that 
\begin{equation}\label{eqn:intro1}
\frac{1}{K_{\mathcal{U}}(\boldsymbol{z}, \boldsymbol{w})} = \sum_{n\geq 0} \hat{b}_n\langle \boldsymbol{z}, \boldsymbol{w}\rangle^n,
\end{equation}
for every $\boldsymbol{z}, \boldsymbol{w} \in \mathbb B_d$. Equation \eqref{eqn:intro1} is one of the main ingredients of the Kaluza lemma which provides a necessary condition for $K_{\mathcal{U}}$ to be CNP (cf. \cite{Agler}). The Kaluza lemma is named after Theodor Kaluza who proved it in terms of power series of one complex variable \cite{KALU}.

\begin{theorem}\label{TK Thm}
Let $\{c_n\}_{n \geq 0}$ be a sequence of positive real numbers with $c_0 = 1$ and $M$ be a positive real number. If the sequence of real numbers $\left\lbrace \frac{c_n}{c_{n-1}} : n > 0 \right\rbrace$ is non-decreasing and bounded above by $M$, then $f(z) = \sum_{n \geq 0} c_n z^n$ converges for all $z \in B(0, 1/M) := \{z \in \mathbb C : |z| < 1/M\}$ and there exists a sequence of non-negative real numbers $\{q_n\}_{n > 0}$ such that 
$$1 - \frac{1}{f(z)} = \displaystyle \sum_{n \geq 1} q_n z^n,\,\,z \in B(0, 1/M).$$
\end{theorem}

A kernel $K : \Omega \times \Omega \to \mathbb C$ is said to be irreducible if $K(\cdot, \boldsymbol{w}_1)$ and $K(\cdot, \boldsymbol{w}_2)$ are linearly independent for any two distinct points $\boldsymbol{w}_1, \boldsymbol{w}_2 \in \Omega$ and $K(\boldsymbol{z}, \boldsymbol{w}) \neq 0$ for every $\boldsymbol{z}, \boldsymbol{w} \in \Omega$. A well-known characterization of an irreducible kernel $K: \Omega \times \Omega \to \mathbb C$ having the CNP property is that $1 - \frac{1}{K}$ must be non-negative definite (cf. \cite[Theorem 7.28]{Agler}). Combining this characterization with Theorem \ref{TK Thm}, we obtain the following necessary condition for a $\mathcal{U}(d)$-invariant kernel to have the CNP property (cf. \cite[Lemma 7.38]{Agler}).

\begin{lemma}\label{kalu lem}
An irreducible $\mathcal{U}(d)$-invariant kernel $K_{\mathcal{U}} = \sum_{ n \geq 0} a_n \langle \boldsymbol{z}, \boldsymbol{w}\rangle^n$ on $\mathbb B_d$ with $a_{n} > 0$ for each $n$ is CNP if 
$$\frac{a_n}{a_{n-1}} \leq \frac{a_{n+1}}{a_n}$$
holds for all $n \geq 1$.
\end{lemma}

%We use this characterization to obtain a necessary and sufficient condition for a $\mathbb K$-invariant kernel $K = \sum_{\s \in \vec{\mathbb N}_0^r} a_{\s} K_{\s}$ to be CNP in terms of the coefficients $a_{\s},\,\,\s \in \vec{\mathbb N}_0^r$. 
%\textit{Therefore, throughout this article, we assume that a  $\mathbb K$-invariant kernel is irreducible.}  

Suppose $K = \sum_{\s \in \vec{\mathbb N}_0^r} a_{\s} K_{\s}$ is a $\mathbb K$-invariant kernel on $\Omega$. In Section 2, we first obtain a sequence of real numbers $\{\hat{b}_{\s}\}_{\s \in \vec{\mathbb N}_0^r}$ such that 
\begin{equation}\label{eqn:intro2}
   \frac{1}{K(\boldsymbol{z}, \boldsymbol{w})} = \displaystyle \sum_{\s \in \vec{\mathbb N}_0^r} \hat{b}_{\s}K_{\s} (\boldsymbol{z}, \boldsymbol{w})
\end{equation}
holds for every $\boldsymbol{z}, \boldsymbol{w} \in \Omega.$ Assuming $K$ to be an irreducible kernel, we generalize Lemma \ref{kalu lem} to find a necessary condition for $K$ to be CNP in terms of the coefficients $a_{\s},\,\,\s \in \vec{\mathbb N}_0^r$. We also provide examples of $\mathbb K$-invariant CNP kernels on $\Omega$ such that the multiplication operators by the coordinate functions on the corresponding Hilbert spaces are bounded.  
%Henceforth, we always assume that a $\mathbb K$-invariant kernel $K = \sum_{\s \in \vec{\mathbb N}_0^r} a_{\s} K_{\s}$ on $\Omega$ is non-vanishing.

Given a contraction $T$ on a Hilbert space $\mathcal{H}$, let $D_T := (I - T^*T)^{\frac{1}{2}}$, $D_{T^*} := (I - TT^*)^{\frac{1}{2}}$ and $\mathcal{D}_{T} := \overline{\mbox{Ran}\,\, D_{T}}$, $\mathcal{D}_{T^*} := \overline{\mbox{Ran}\,\, D_{T^*}}$. The operators $D_T$, $D_{T^*}$ are called the defect operators and the spaces $\mathcal{D}_{T}$, $\mathcal{D}_{T^*}$ are called the defect spaces of $T$. The operator valued holomorphic function $\theta_T : \mathbb D \to \mathcal{B}(\mathcal{D}_{T}, \mathcal{D}_{T^*})$, defined by 
$$\theta_T(z) = \left(-T + zD_{T^*}(I - zT^*)^{-1}D_T\right)|_{\mathcal{D}_{T}},\,\,z \in \mathbb D,$$
is called the characteristic function of $T$. The classical Sz.-Nagy--Foias theory says that two completely nonunitary (cnu) contractions $T_1$ and $T_2$ are unitarily equivalent if and only if the characteristic functions $\theta_{T_1}$ and $\theta_{T_2}$ coincide (that is, there exist two unitaries $u : \mathcal{D}_{T_1} \to \mathcal{D}_{T_2}$ and $v : \mathcal{D}_{T_1^*} \to \mathcal{D}_{T_2^*}$ such that $v\theta_{T_1}(z) = \theta_{T_2}(z)u,\,\,z \in \mathbb D$) (cf. \cite[Theorem 3.4]{Na-Fo}).

Note that $T$ is a contraction if and only if $\frac{1}{\mathbb S}(T, T^*) := I - TT^*$ is a positive operator, where $\mathbb S : \mathbb D \times \mathbb D \to \mathbb C$ is the Szeg\H{o} kernel $\mathbb S(z, w) = \frac{1}{1 - z\bar w},\,\,z, w \in \mathbb D$. Let $K_{D} : \mathbb B_d \times \mathbb B_d \to \mathbb C$ be defined by $K_D(\boldsymbol{z}, \boldsymbol{w}) = \frac{1}{1 - \langle \boldsymbol{z}, \boldsymbol{w}\rangle},\,\,\boldsymbol{z}, \boldsymbol{w} \in \mathbb B_d$. The kernel $K_D$ is called the Drury-Arveson kernel. A commuting tuple of operators $\boldsymbol{T} = (T_1, \ldots, T_d)$ is a $\frac{1}{K_D}$-contraction if and only if $\frac{1}{K_D}(\boldsymbol{T}, \boldsymbol{T}^*) = I - \sum_{i=1}^d T_i T_i^*$ is a positive operator. Such a tuple of operators is called a row contraction. For a commuting row contraction $\boldsymbol{T}$, the characteristic function $\theta_{\boldsymbol{T}} : \mathbb B_d \to \mathcal{B}(\mathcal{D}_{\boldsymbol{T}}, \mathcal{D}_{\boldsymbol{T}^*})$ of $\boldsymbol{T}$ is defined in a similar fashion to the classical characteristic function of a single contraction, where $\mathcal{D}_{\boldsymbol{T}}$ and $\mathcal{D}_{\boldsymbol{T}^*}$ are defect spaces introduced in the usual way. The characteristic function of a pure commuting row contraction $\boldsymbol{T}$ determines the unitary equivalence class of $\boldsymbol{T}$ (cf. \cite{Tiret2}). 

Given a non-vanishing $\mathcal{U}(d)$-invariant kernel $K_{\mathcal{U}}(\boldsymbol{z}, \boldsymbol{w}) = \sum_{n \geq 0} a_n \langle \boldsymbol{z}, \boldsymbol{w}\rangle^n$ on $\mathbb B_d$  and a commuting $d$-tuple of operators $\boldsymbol{T} = (T_1, \ldots, T_d)$ on a Hilbert space $\mathcal{H}$, the $\frac{1}{K_{\mathcal{U}}}$-contractivity of $\boldsymbol{T}$ is defined in a similar fashion of a $\frac{1}{K_D}$-contraction due to Equation \eqref{eqn:intro1} (\cite[Definition 1.1]{Tiret}). A characterization of $\mathcal{U}(d)$-invariant CNP kernels over $\mathbb B_d$ is obtained in terms of the existence of the characteristic function of $\frac{1}{K_{\mathcal{U}}}$-contractions in \cite{Tiret}. In particular, if $K_{\mathcal{U}}$ also has the CNP property, then an explicit construction of the characteristic function of a $\frac{1}{K_{\mathcal{U}}}$-contraction is given in the final section of \cite{Tiret} and it is proved that the characteristic function of a pure $\frac{1}{K_{\mathcal{U}}}$-contraction determines it's unitary equivalence class, see \cite{Tiret}. 

Suppose $K = \sum_{\s \in \vec{\mathbb N}_0^r} a_{\s} K_{\s}$ is a $\mathbb K$-invariant kernel on $\Omega$. In section 3, we provide a similar definition of $\frac{1}{K}$-contractivity of a commuting tuple of operators $\boldsymbol{T} = (T_1, \ldots, T_d)$ on a Hilbert space $\mathcal{H}$. The main objective of this section is to obtain a necessary and sufficient condition of a $\mathbb K$-invariant kernel on $\Omega$ to have the CNP property in terms of the existence of the characteristic function of $\frac{1}{K}$-contractions. In the final section, we assume that $K$ aslo has the CNP property. Then, we provide an explicit construction of the characteristic function of a $\frac{1}{K}$-contraction $\boldsymbol{T}$ for certain $\mathbb K$-invariant kernels. We also prove that, for such kernels, the unitary equivalance class of a pure$\frac{1}{K}$-contraction is determined by the characteristic function.

\section{A generalization of Kaluza lemma for $\mathbb K$-invariant kernels}

%In this section, we generalize the well-known Kaluza's lemma providing a criteria of a non-vanishing $\mathbb K$-invariant kernel to have CNP property.

In this section, we provide a criteria of a non-vanishing $\mathbb K$-invariant kernel $K = \sum_{\s} a_{\s} K_{\s}$ on $\Omega$
to have CNP property in terms of the sequence $\{a_{\s}\}_{\s \geq \underline 0}$. We begin this section with a criterion of convergence of a $\mathbb K$-invariant function defined on $\Omega \times \Omega$ which is holomorphic in the first $d$-variables and anti-holomorphic in the last $d$-variables.

\begin{proposition} \label{pp1}
Let 
\begin{equation}\label{sA}
    F(\boldsymbol{z}, \boldsymbol{w}) := \sum_{\s} c_{\s} K_{\s}(\boldsymbol{z}, \boldsymbol{w}),
\end{equation}
be a $\mathbb K$-invariant function on $\Omega\times\Omega$ and $e$ denote any maximal tripotent of $\Omega$.

{\rm (i)} If $F$ converges absolutely on the closure of $\Omega\times\Omega$ and is bounded in modulus by $1$, then
\begin{equation} \label{sB}
  \sum_{\s} |c_{\s}| K_{\s}(e,e) \leq 1.  
\end{equation}
Conversely, if \eqref{sB} holds, then $F$ converges absolutely and uniformly on the closure of $\Omega\times\Omega$, and is bounded in modulus by $1$ there.

{\rm (ii)} If $F$ converges absolutely on $\Omega\times\Omega$, then
\begin{equation}\label{sC}
S(t) := \sum_{\s} |c_{\s}| K_{\s}(te,te) < \infty,    
\end{equation}
for all $t \in (0, 1)$.
Conversely, if \eqref{sC} holds, then $F$ converges absolutely and uniformly on compact subsets of $\Omega\times\Omega$.

\end{proposition}

\begin{proof} (i) Taking $z=w=e$ shows that the assumption on \eqref{sA} implies \eqref{sB}.

For the converse part, recall that a function holomorphic on $\Omega$ and continuous on its closure attains its maximum modulus on the Shilov boundary, and the Shilov boundary of $\Omega$ is precisely the orbit $\{ke: k\in \mathbb K\}$ of a maximal tripotent $e$ (see \cite{Arazy}). Hence
$$ |K_{\s}(\boldsymbol{z}, \boldsymbol{w})| \leq |K_{\s}(k_1e,k_2e)| $$
for all $\boldsymbol{z}, \boldsymbol{w} \in \overline \Omega$ and for some $k_1,k_2\in\mathbb K$. Now by Cauchy-Schwarz
$$ |K_{\s}(k_1e,k_2e)|^2 \leq K_{\s}(k_1e,k_1e) K_{\s}(k_2e,k_2e) = K_{\s}(e,e)^2  $$
thanks to the $\mathbb K$-invariance of $K_{\s}$. Consequently,
$$ \sum_{\s} |c_{\s} K_{\s}(\boldsymbol z, \boldsymbol w)| \leq \sum_{\s} |c_{\s}| K_{\s}(e,e) \,\,\forall \boldsymbol z, \boldsymbol w\in\overline\Omega, $$
and the claim follows.

(ii) Let $t \in (0, 1)$. Suppose $M_t > 0$ is such that $|F(t\boldsymbol{z}, t\boldsymbol{w})| \leq M_t$ for all $\boldsymbol{z}, \boldsymbol{w} \in \Omega$. Now, apply part (i) to the functions $F(t\boldsymbol{z},t\boldsymbol{w})/M_t$.
\end{proof}

The following proposition provides an explicit description of a $\mathbb K$-invariant function on $\Omega \times \Omega$, holomorphic in $\boldsymbol{z}, \bar{\boldsymbol{w}}$. The proof is modeled on the proof of \cite[Proposition 2]{AEq}.

\begin{proposition} \label{pp2}
If $F(\boldsymbol{z}, \boldsymbol{w})$ is a $\mathbb K$-invariant function on $\Omega\times\Omega$ holomorphic in $(\boldsymbol{z},\bar{\boldsymbol{w}})$, then it has an expansion
\begin{equation}\label{sD}
   F(\boldsymbol z, \boldsymbol w) = \sum_{\s} c_{\s} K_{\s}(\boldsymbol z, \boldsymbol w) 
\end{equation}
with some $c_{\s}$ satisfying \eqref{sC}.
\end{proposition}

\begin{proof} 
Assume first that $F$ extends continuously up to the boundary of $\Omega\times\Omega$, and fix a maximal tripotent $e$. The holomorphic function $F(\boldsymbol{z},e)=: F_e(\boldsymbol{z})$, being bounded, then belongs to the ordinary (unweighted) Bergman space on $\Omega$, for which $\{\psi^{\s}_{\alpha}/\sqrt{(p)_{\s}}\}_{\alpha,\s}$ ($p$~being the genus of $\Omega$) serves as an orthonormal basis. Thus the Fourier decomposition
$$ F_e = \sum_{\s} \sum_\alpha \frac{\langle F_e,\psi^{\s}_{\alpha}\rangle}{(p)_{\s}} \psi^{\s}_{\alpha} =: \sum_{\s} f_{\s}$$
converges in $L^2(\Omega)$ and, by the properties of the reproducing kernels \cite{Aro}, absolutely and uniformly on compact subsets of $\Omega$. If $\mathbb L$ denotes the stabilizer subgroup of $e$ in $\mathbb K$, then for any $l\in\mathbb L$ we have, by $\mathbb K$-invariance,
$F_e(l\boldsymbol z)=F(l\boldsymbol z,e)=F(\boldsymbol z,l^{-1}e)=F(\boldsymbol z,e)=F_e(\boldsymbol z)$; in other words, $F_e$  and, hence, each $f_{\s}$ are $\mathbb L$-invariant. However, the only $\mathbb L$-invariant element in $\mathcal P_{\s}$, up to constant multiples, is $K_{\s}(\cdot,e)$ \cite[Theorem~2.1]{FK}. Thus

$$ F(\boldsymbol z,e) = \sum_{\s} c_{\s} K_{\s}(\boldsymbol z,e)\,\, \forall \boldsymbol z\in\Omega,   $$
with uniform convergence on compact subsets. By $\mathbb K$-invariance of both $F$ and $K_{\s}$, this equality remains in force also for $F(\boldsymbol z,ke)=F(k^{-1}\boldsymbol z,e)$, uniformly for $\boldsymbol z$ in compact subsets of $\Omega$ and $k\in\mathbb K$. Since, for each fixed $\boldsymbol z$, the anti-holomorphic functions $F(\boldsymbol z, \boldsymbol w)$ and $K_{\s}(\boldsymbol z, \boldsymbol w)$ of $\boldsymbol w$ attain their maximum moduli on the Shilov boundary $\{ke:k\in\mathbb K\}$, it transpires by the same argument as in the proof of Proposition \ref{pp1} that \eqref{sD} holds, with convergence absolute and uniform for $\boldsymbol z, \boldsymbol w$ in compact subsets of $\Omega$. (In fact, $\boldsymbol w$ could even be taken in the closure of $\Omega$.) By part (ii) of Proposition \ref{pp1}, the claim follows.
 
For general $F$ (i.e.not necessarily continuous up to the boundary), apply the argument above to $F(t\boldsymbol{z},t\boldsymbol{w})$ for each $t\in(0,1)$.
\end{proof}

The last proposition applies, in particular, to $F=K$ with any $\mathbb K$-invariant kernel $K$, and to $F=1-1/K$ with any such $K$ which is in addition zero-free. If $K(\boldsymbol{z}, \boldsymbol{w}) = \sum_{\s} a_{\s} K_{\s}(\boldsymbol{z}, \boldsymbol{w})$ is a $\mathbb K$-invariant kernel on $\Omega \times \Omega$, then there exists a sequence of real numbers $\{b_{\s}\}$ with $b_{\underline 0} = 0$ such that $1 - \frac{1}{K} = \sum_{\s > \underline 0} b_{\s}K_{\s}$. 
%Rewriting it, we obtain $\frac{1}{K} = \sum_{\s} \hat{b}_{\s} K_{\s}$, where $\hat{b}_{\underline 0} = 1$ and $\hat{b}_{\s} = -b_{\s}$ for every $\s > \underline 0$. 
%Thus, we have
% \begin{equation}\label{eqn:b_s}
% \left(\sum a_{\s} K_{\s} \right) \left( \sum \hat{b}_{\s} K_{\s} \right) = 1. 
% \end{equation}
% Let $\s_0$ be a signature of length $n \geq 1$. Note that 
% \begin{equation}\label{exist:eqn1 b_s}
%   K_{\s}K_{\tilde{\s}} = \sum_{|\underline p| = |\s| + |\tilde{\s}|} c^{\underline p}_{\s, \tilde{\s}} K_{\underline p},  
% \end{equation}
% where the coefficients $c^{\underline p}_{\s, \tilde{\s}} \geq 0$ and $\s_{i}, \tilde{\s}_i \leq \underline p_i$ for every $1 \leq i \leq r$ (cf. \cite[Equation 2.9]{Arazy2}). Now, using Equation \eqref{exist:eqn1 b_s} and then comparing the coefficients of $K_{\s_0}$ from both sides of Equation \eqref{eqn:b_s}, we obtain
% $$\displaystyle \sum_{m=0}^n \sum_{|\s| = n - m} \sum_{|\s'| = m} a_{\s} \hat{b}_{\s'} c_{\s, \s'}^{\s_0} = 0,$$
% and therefore, 
% $$a_{\underline 0} \hat{b}_{\s_0} = - \displaystyle \sum_{m=0}^{n-1} \sum_{|\s| = n - m} \sum_{|\s'| = m} a_{\s} \hat{b}_{\s'} c_{\s, \s'}^{\s_0}.$$
% It follows from mathematical induction on $n$ that $\hat{b}_{\s_0}$ is obtained from $\hat{b}_{\s'}$, $|\s'| < n$.

\begin{corollary} \label{pp3}
Any $F$ as is the last proposition satisfies
$$ F(t\boldsymbol z, \boldsymbol w)=F(\boldsymbol z,t \boldsymbol w)\,\,\forall \boldsymbol z, \boldsymbol w\in\Omega $$
for any $0<t<1$. In~particular, for each $\boldsymbol w\in\Omega$, $F(\cdot, \boldsymbol w)$ extends to a function holomorphic in a neighborhood of the closure of~$\Omega$.
\end{corollary}

\begin{proof} 
For each $\s$ and any $0<t<1$, $K_{\s}(t\boldsymbol z, \boldsymbol w)=t^{|\s|}K_{\s}(\boldsymbol z, \boldsymbol w)=K_{\s}(\boldsymbol z,t \boldsymbol w)$.
\end{proof}

Now we have all the necessary tools to formulate a criterion for a non-vanishing $\mathbb K$-invariant kernel $K = \sum_{\s} a_{\s} K_{\s}$ on $\Omega$ to posses the CNP property in terms of the sequence $\{a_{\s}\}_{\s \geq \underline 0}$. We first observe that every non-vanishing $\mathbb K$-invariant kernel is irreducible. 

\begin{lemma}\label{irreducibility}
Let $K = \sum_{\s} a_{\s}K_{\s}$ be a non-vanishing $\mathbb K$-invariant kernel on $\Omega$. Then $K$ is irreducible.    
\end{lemma}

\begin{proof}
Note that it is enough to prove $K(\cdot, \boldsymbol{w}_1)$ and $K(\cdot, \boldsymbol{w}_2)$ are linearly independent for any two distinct points $\boldsymbol{w}_1, \boldsymbol{w}_2 \in \Omega$. Let $\boldsymbol{w}_1$ and $\boldsymbol{w}_2$ be two distinct elements of $\Omega$. Suppose $\boldsymbol{w}_j(i)$, $j = 1, 2$, $1 \leq i \leq r,$ represents the $i$th component of the point $\boldsymbol{w}_j$. Assume that there exists a $c \in \mathbb C$ such that 
\begin{equation}\label{eqn:irr1}
K(\cdot, \boldsymbol{w}_1) = c K(\cdot, \boldsymbol{w}_2).    
\end{equation}
Evaluating both sides of Equation \eqref{eqn:irr1} at $\boldsymbol{0} \in \Omega$, we obtain $c = 1$. Substituting $c = 1$ and the expression of $K$ in Equation \eqref{eqn:irr1}, we get 
\begin{flalign}\label{eqn:irr2}
\nonumber 0 &= \displaystyle \sum_{\s \in \vec{\mathbb N}_0^r, |\s| \geq 1} a_{\s} (K_{\s}(\boldsymbol{z}, \boldsymbol{w_1}) - K_{\s}(\boldsymbol{z}, \boldsymbol{w_2}))\\
&=  a_{\epsilon_1} \langle \boldsymbol{z}, \boldsymbol{w}_1 - \boldsymbol{w_2}\rangle + \displaystyle \sum_{\s \in \vec{\mathbb N}_0^r, |\s| \geq 2} a_{\s} \sum_{\alpha = 1}^{d_{\s}} \psi_{\alpha}^{\s}(\boldsymbol{z}) (\psi_{\alpha}^{\s}(\boldsymbol{w_1}) - \psi_{\alpha}^{\s}(\boldsymbol{w_2})),\,\,\boldsymbol{z} \in \Omega.
\end{flalign}
Recall that $a_{\epsilon_1} \neq 0$. Thus, differentiating both sides of Equation \eqref{eqn:irr2} with respect to the $i$th variable and then evaluating at $\boldsymbol{z} = \boldsymbol{0}$, we conclude that $\boldsymbol{w}_{1}(i) = \boldsymbol{w}_2(i)$. This proves that $\boldsymbol{w}_1 = \boldsymbol{w}_2$. Hence, $K$ is irreducible. 
\end{proof}

Henceforth, we consider only non-vanishing $\mathbb K$-invariant kernels on $\Omega$. The following corollary is an immediate consequence of Proposition \ref{pp2} which provides a criterion of a $\mathbb K$-invariant kernel to have CNP property.

\begin{corollary}\label{k-inv-cnp}
A non-vanishing $\mathbb K$-invariant kernel $K: \Omega \times \Omega \to \mathbb C$ is CNP if and only if there exists a sequence of non-negative real numbers $\{b_{\underline{s}}\}_{\underline s > \underline 0}$ such that 
\begin{equation}\label{eqn:1-1/k}
 1 - \frac{1}{K(\boldsymbol{z}, \boldsymbol{w})} = \displaystyle \sum_{\underline{s} > \underline 0} b_{\underline{s}} K_{\underline{s}} (\boldsymbol{z}, \boldsymbol{w})   
\end{equation}
holds for every $\boldsymbol{z}, \boldsymbol{w} \in \Omega$.
\end{corollary}

\begin{proof}
If $K$ is a non-vanishing $\mathbb K$-invariant kernel, then the existence of a sequence of real numbers $\{b_{\s}\}_{\s > \underline 0}$ follows from Proposition \ref{pp2}. Note that $K$ is irreducible, thanks to Lemma \ref{irreducibility}. Due to \cite[Theorem 7.28]{Agler}, it follows that $K$ is a CNP kernel if and only if $1 - \frac{1}{K}$ is non-negative definite. Also, note that $\sum_{\underline{s} > 0} b_{\underline{s}} K_{\underline{s}}$ is a non-negative definite kernel if and only if each $b_{\s} \geq 0$ (cf. \cite{FK}). 
\end{proof}

\begin{remark} \label{remME2}

The last corollary even holds for general nonvanishing $\mathbb K$-invariant functions~$K$, i.e. it is not necessary to assume that $K(\boldsymbol z,\boldsymbol w)$ is a nonnegative-definite kernel ---the nonnegative definiteness follows automatically. Indeed, if $1-1/K=:L$ is nonnegative-definite, then
$$ K = \frac1{1-L} = 1 + L + L^2 + L^3 + \dots $$
must also be nonnegative-definite: any power $L^k$, $k=0,1,2,\dots$, is~a nonnegative definite kernel by \cite[Section~I.8]{Aro}, hence also any finite sum of these powers \cite[Section~I.6]{Aro}, and, finally, their increasing limit \cite[Section~I.9.B]{Aro}. (For the last, observe that $K$ being nonvanishing implies that $0\leq L(\boldsymbol z,\boldsymbol z)<1$ for all $\boldsymbol z\in\Omega$.)
\end{remark}

Note that a non-vanishing unitary invariant kernel $K_{\mathcal{U}}(\boldsymbol{z}, \boldsymbol{w}) = \sum_{n\geq 0} a_n \left\langle \boldsymbol{z}, \boldsymbol{w}\right\rangle^n$
on the unit ball $\mathbb B_d$ has CNP property if 
$$\frac{a_{n+1}}{a_{n}} \geq \frac{a_{n}}{a_{n-1}}$$
holds for every $n \geq 1$. This is the classical Kaluza lemma \cite[Lemma 7.38]{Agler}. In the following theorem, we generalize Kaluza lemma providing a criteria for a $\mathbb K$-invariant kernel $K = \sum_{\s \geq \underline 0} a_{\s}K_{\s}$ on $\Omega$ to have the CNP property in terms of the sequence $\{a_{\s}\}_{\s \geq \underline 0}$.

\begin{theorem}[Generalized Kaluza lemma]\label{genKaluza}
Let $K = \sum a_{\s} K_{\s}$ be a $\mathbb K$-invariant kernel. Then $K$ is a CNP kernel if for each signature $\s_0$ of length $k \geq 1$, the following holds
\begin{equation}\label{eqn:Kaluza}
\displaystyle \sum_{\tilde{\s}_0 \in I(\s_0)} \sum_{|\underline{\tilde{p}}| = k - 1 - |\underline{q}|} \frac{a_{\underline{\tilde{p}}}}{a_{\tilde{\s}_0}} c^{\tilde{\s}_0}_{\underline{\tilde{p}}, \underline{q}} \geq |I(\s_0)| \displaystyle \sum_{|\underline p| = k - |\underline q|} \frac{a_{\underline p}}{a_{\s_0}} c_{\underline p, \underline q}^{\s_0},    
\end{equation}
where $\underline q$ is any signature with $1 \leq | \underline q| \leq k-1$,  $I(\s_0) = \{\s_0 - \epsilon_i : \s_0 - \epsilon_i\,\,\mbox{is a signature}\}$ and $|I(\s_0)|$ is the cardinality of the set $I(\s_0)$.
\end{theorem}

\begin{proof}
Due to \cite[Theorem 7.28]{Agler}, $K$ is a CNP kernel if and only if $1 - \frac{1}{K}$ is non-negative definite. It follows from Proposition \ref{pp2} that there exists a sequence of real numbers $\{\hat{b}_{\s}\}_{\s \geq \underline 0}$ with $\hat{b}_{\underline 0} = 1$ such that $\frac{1}{K} = \sum \hat{b}_{\s}K_{\s}$. This implies that 
$$1 - \frac{1}{K} = - \displaystyle \sum_{\s > \underline 0} \hat{b}_{\s}K_{\s}.$$
Thus, $1 - \frac{1}{K}$ is non-negative definite if each $\hat{b}_{\s},\,\,\s > \underline 0,$ is a non-positive real number. Consequently, existence of a sequence of non-positive real numbers $\{\hat{b}_{\s}\}_{\s > \underline 0}$ such that $\frac{1}{K} = 1 + \sum_{\s>\underline 0} \hat{b}_{\s} K_{\s}$ is equivalent to the fact that $K$ is a CNP kernel. Therefore, to prove the theorem, we have to prove that the existence of a sequence $\{\hat{b}_{\s}\}_{\s \geq \underline 0}$ with $\hat{b}_{\underline 0} = 1$ and $\hat{b}_{\s} \leq 0,\,\,\s > \underline 0$ such that the equation
\begin{equation}\label{eqn1:Kaluza}
\left( \displaystyle \sum_{\s \geq \underline 0} a_{\s}K_{\s}\right) \left( \displaystyle \sum_{\s \geq \underline 0} \hat{b}_{\s}K_{\s} \right) = 1    
\end{equation}
holds. Comparing the coefficient of $K_{\epsilon_1}$ from both sides of Equation \eqref{eqn1:Kaluza}, we obtain 
\begin{equation}\label{eqnx:sec1:kalu}
 \hat{b}_{\epsilon_1} = -a_{\epsilon_1}.   
\end{equation}

%Assume that if Equation \eqref{eqn:Kaluza} holds for every signature $\title{\s}$ of leght less than $k$, then, comparing coefficient of $K_{\tilde{\s}}$ from both sides of Equation \eqref{eqn1:Kaluza}, we obtain $b_{\tilde{\s}} \leq 0$.
For any two signatures $\s$ and $\tilde{\s}$, note that 
\begin{equation}\label{exist:eqn1 b_s}
  K_{\s}K_{\tilde{\s}} = \sum_{|\underline p| = |\s| + |\tilde{\s}|} c^{\underline p}_{\s, \tilde{\s}} K_{\underline p},  
\end{equation}
 where the coefficients $c^{\underline p}_{\s, \tilde{\s}} \geq 0$, $\underline p \geq \s$ and $\underline p \geq \tilde{\s}$ (cf. \cite[Equation 2.9]{Arazy2}).

Let $k \geq 1$. Suppose $\s_0$ is a signature of length $k$. Comparing the coefficients of $K_{\s_0}$ from both sides of Equation \eqref{eqn1:Kaluza}, we get
\begin{equation}\label{eqn2:Kaluza}
\hat{b}_{\s_0}a_{\underline 0} + \displaystyle \sum_{0 < |\underline q| < k}  \left( \sum_{|\underline p| = k - |\underline q|} a_p c^{\s_0}_{\underline p, \underline q} \right) \hat{b}_{\underline q} + \hat{b}_{\underline 0} a_{\s_0} = 0.  
\end{equation}
Note that $|\tilde{\s}_0| = k - 1$ for every $\tilde{\s}_0 \in I(\s_0)$. Therefore, for each $\tilde{\s}_0 \in I(\s_0)$, comparing the coefficient of $K_{\tilde{\s}_0}$ from both sides of Equation \eqref{eqn1:Kaluza}, we obtain
\begin{equation}\label{eqn3:Kaluza}
\hat{b}_{\tilde{\s}_0}a_{\underline 0} + \displaystyle \sum_{0 < |\underline q| < k-1}  \left( \sum_{|\underline{\tilde{p}}| = k - 1 - |\underline{q}|} a_{\underline{\tilde{p}}} c^{\tilde{\s}_0}_{\underline{\tilde{p}}, \underline{q}} \right) \hat{b}_{\underline q} + \hat{b}_{\underline 0} a_{\tilde{\s}_0} = 0.  
\end{equation}
Multiplying Equation \eqref{eqn3:Kaluza} by $\frac{a_{\s_0}}{|I(\s_0)| a_{\tilde{\s}_0}}$, adding them up for all $\tilde{\s}_0 \in I(\s_0)$ and then subtracting it from Equation \eqref{eqn2:Kaluza}, we obtain
\begin{equation}\label{eqn4:Kaluza}
\hat{b}_{\s_0}a_{\underline 0} = \displaystyle \sum_{0 < |\underline q| \leq k - 1} \left[ \sum_{\tilde{\s}_0 \in I(\s_0)} \left(\sum_{|\underline{\tilde{p}}| = k - 1 - |\underline{q}|} \frac{a_{\underline{\tilde{p}}} a_{\s_0}}{|I(\s_0)| a_{\tilde{\s}_0}} c^{\tilde{\s}_0}_{\underline{\tilde{p}}, \underline{q}} \right) - \sum_{|\underline{p}| = k - |\underline{q}|} a_{\underline{p}} c^{\s_0}_{\underline{p}, \underline{q}} \right] \hat{b}_{\underline q}.   
\end{equation}
% Note that for every signatures $\underline{\tilde{p}}$ and $\underline q$, there exists $\tilde{\s}_0 \in I(\s_0)$ such that $C^{\tilde{\s}_0}_{\underline{\tilde{p}}, \underline q} > 0.$ Thus, we have 
% $$\sum_{\tilde{\s}_0 \in I(\s_0)} \left(\sum_{|\underline{\tilde{p}}| = k - 1 - |\underline{q}|} \frac{a_{\underline{\tilde{p}}} a_{\s_0}}{|I(\s_0)| a_{\tilde{\s}_0}} c^{\tilde{\s}_0}_{\underline{\tilde{p}}, \underline{q}} \right) > 0$$
% and hence, none of the coefficients of $\hat{b}_{\underline q}$ in Equation \eqref{eqn4:Kaluza} is automatically negative. 
Therefore, it follows from Equation \eqref{eqn4:Kaluza} that $\hat{b}_{\s_0} \leq 0$ if Equation \eqref{eqn:Kaluza} holds for every signatures $\underline q$ with $1 \leq |\underline q| \leq k-1$.
\end{proof}

\begin{remark}
%In particular, if we consider the rank of the domain $\Omega$ to be $1$, that is, $\Omega$ is the unit ball $\mathbb B_d$, and $\mathbb K$ to be $\mathcal{U}(d)$, then we retrieve the Kaluza's lemma (cf. \cite[Lemma 7.38]{Agler}).    
Denoting momentarily
$$ A^{\underline s}_{\underline q} := \sum_{\underline p} \frac{a_{\underline p}}{a_{\underline s}} c^{\underline s}_{\underline p,\underline q}  $$
(note that $c^{\underline s}_{\underline p,\underline q}$ is nonzero only for $|\underline s|=|\underline p|+|\underline q|$, so~the last sum is automatically finite), Equation \eqref{eqn:Kaluza} can be rephrased~as
$$\frac{1}{I(\s_0)} \displaystyle \sum_{\tilde{\s}_0 \in I(\s_0)} A^{\underline s_0}_{\underline q} \geq A_{\underline q}^{\s_0},\,\\forall \,\,|\underline q| \geq 1.$$
That~is, $A^{\underline s_0}_{\underline q}$ has to be less than or equal to the average of the same quantities over $I(\underline s_0)$.

For $\Omega$ the ball, $I(\underline s_0)$ becomes a singleton, $A^{\underline s}_{\underline q}$ becomes just $a_{s-q}/a_s$, and the connection with the ordinary Kaluza Lemma becomes evident.
\end{remark}

\subsection{Examples}
The space of all the square integrable holomorphic functions on $\Omega$ $\mathbb A^{2}(\Omega)$ with respect to the Lebesgue measure $dm(z)$ is known as the Bergman space over $\Omega$. It is a reproducing kernel Hilbert space. The reproducing kernel of $\mathbb A^{2}(\Omega)$, also known as the Bergman kernel, is $B(\boldsymbol z, \boldsymbol w)=\Delta (\boldsymbol{z}, \boldsymbol{w})^{-p},$ where $p=2+a(r-1)+b,$ is called the {\it genus} of the domain $\Omega$ (see \cite[Theorem 2.9.8 ]{upmeier}). 
Here $\Delta(\boldsymbol{z}, \boldsymbol{w})$ is a $\mathbb K$-invariant sesqui-analytic polynomial on $\Omega \times \Omega,$ 
%which is uniquely determined by the property $\Delta(\boldsymbol{z}, \boldsymbol{z})=\prod_{i=1}^r(1-t^2_i).$ The function $\Delta(\boldsymbol{z}, \boldsymbol{w})$ is 
 known as the {\it Jordan triple determinant}. 

Also, for $\nu \in \mathcal{W}_{\Omega} := \left \{0, \ldots, \frac{a}{2}(r-1)\right \}\cup \left( \frac{a}{2}(r-1), \infty\right)$, the sesqui-analytic function $K^{(\nu)}:\Omega \times \Omega \to \mathbb C$ defined by
\beq \label{FK formula} K^{(\nu)} (\boldsymbol z, \boldsymbol w):= \Delta(\boldsymbol z, \boldsymbol w)^{-\nu}= \sum_{\s} (\nu)_{\s}K_{\s}(\boldsymbol z, \boldsymbol w),\;\; \boldsymbol z, \boldsymbol w \in \Omega,\eeq is non negative definite  (see \cite[Corollary 5.2]{FK} ). Here, $(\nu)_{\s}$ denotes the {\it generalized Pochhammer symbol}
\[(\nu)_{\s}:= \prod_{j=1}^{r} \left( \nu-\frac{a}{2}(j-1)\right)_{s_{j}}= \prod_{j=1}^{r}\prod_{l=1}^{s_{j}} \left(\nu-\frac{a}{2}(j-1)+l-1\right).\]
Therefore, by the Moore-Aronszajn Theorem, $K^{(\nu)}$ determines a Hilbert space $\mathbb{A}^{(\nu)}(\Omega)$, known as the weighted Bergman space. If $\nu= \frac{d}{r}$ and $\nu=\frac{a}{2}(r-1)+\frac{d}{r} + 1,$ then the weighted Bergman spaces $\mathbb A^{(\nu)}(\Omega)$ coincide with the Hardy space over the Shilov boundary of $\Omega$ and the classical Bergman space, respectively.

\begin{proposition}\label{BermanNotCNP}
Let $\Omega$ be a Cartan domain of rank $r > 1$. For any $\nu \in \mathcal{W}_{\Omega}$, the weighted Bergman kernel $K^{(\nu)}(\boldsymbol z, \boldsymbol w):= \Delta(\boldsymbol{z}, \boldsymbol{w})^{-\nu}$ is a CNP kernel if and only if $\nu = 0$.    
\end{proposition}

\begin{proof} By~the criterion recalled above, $\Delta(\mathbf z,\mathbf w)^{-\nu}$ being a CNP kernel is equivalent to $1-\Delta(\mathbf z,\mathbf w)^\nu$ being a positive semi-definite kernel. By~the Faraut-Koranyi binomial formula \eqref{FK formula}, this in turn is equivalent to

$$ (-\nu)_{\underline s} \le0 \qquad \forall |\underline s|\ge1.  $$

For $\underline s=(1,0,\ldots, 0)$ and $\underline s=(1,1, 0,\ldots,0)$, this becomes

$$ -\nu\le0 \quad \text{and} \quad (-\nu)(-\nu-\tfrac a2)\le0 ,  $$
respectively. The first means that $\nu\ge0$, while the second means that $-\frac a2\le\nu\le0$ (note that the characteristic multiplicity $a$ is always a nonnegative integer). Hence $\nu=0$ is the only possible solution.
\end{proof}

Hence, it is natural to ask whether there exists any $\mathbb K$-invariant CNP kernel on $\Omega$. Indeed, there exist $\mathbb K$-invariant CNP kernels. In the following, we provide examples of such kernels.

% (i) For $0 < c < \frac{1}{r}$, let $K_c : \Omega \times \Omega \to \mathbb C$ be defined by 
% $$K_c(\boldsymbol{z}, \boldsymbol{w}) = \frac{1}{1 - c \left\langle \boldsymbol{z}, \boldsymbol{w}\right\rangle}, \boldsymbol{z}, \boldsymbol{w} \in \Omega.$$ 
% %The multiplication operators by the coordinate functions are bounded on the Hilbert space $\mathcal{H}_c$ with the reproducing kernel $K_c$. 
% Due to \cite[Theorem 7.28]{Agler}, $K_c$ is a CNP kernel for each $0 < c < \frac{1}{r}$.
Let $\{\hat{b}_{\s}\}_{\s\geq \underline 0}$ be a sequence of real numbers with $\hat{b}_{\underline 0} = 1$ and $\hat{b}_{\s} \leq 0$ for all $\s > \underline 0$. Also, assume that the series $\tilde{L} = \sum_{\s \geq \underline 0} \hat{b}_{\s} K_{\s}$ is convergent uniformly on compact subsets of $\mathbb B_R \times \mathbb B_R$ and non-zero everywhere on $\mathbb B_R \times \mathbb B_R$ where $\mathbb B_R$ is the open ball centered at $0$, radius $R>0$ in $\mathbb C^d$. For $c = \frac{r}{R}$, the function $L(\boldsymbol{z}, \boldsymbol{w}) = \tilde{L}(c \boldsymbol{z}, c \boldsymbol{w})$, $\boldsymbol{z}, \boldsymbol{w} \in \mathbb B_R$, defines a function which is holomorphic in $\boldsymbol{z}$ and anti-holomorphic in $\boldsymbol{w}$ on $\mathbb B_r$. Also, note that $L$ is non-zero everywhere on $\mathbb B_r$. Recall from Equation \eqref{ortho} that for each signature $\s$, 
$$K_{\s} = \displaystyle \sum_{\alpha=1}^{d_{\s}} \psi_{\alpha}^{\s} \overline{\psi_{\alpha}^{\s}},$$ 
where every $\psi_{\alpha}^{\s}$ is a homogeneous polynomial of degree $|\s|.$ Hence, it follows that 
$$L(\boldsymbol{z}, \boldsymbol{w}) = \displaystyle \sum_{\s \geq \underline 0} \hat{b}_{\s}c^{2|\s|} K_{\s} (\boldsymbol{z}, \boldsymbol{w}),\,\,\boldsymbol{z}, \boldsymbol{w} \in \mathbb B_r.$$
Using mathematical induction and Equation \eqref{eqn2:Kaluza}, we obtain a sequence of positive real numbers $\{a_{\s}\}_{\s \geq \underline 0}$, $a_{\underline 0} = 1$, such that Equation \eqref{eqn1:Kaluza} holds. In fact, for every $\s_0 > \underline 0$, Equation \eqref{eqn2:Kaluza} provides a positive real number $a_{\s_0}$. Thus, 
$$K := \frac{1}{L} = \displaystyle \sum_{\s \geq \underline 0} a_{\s} K_{\s}$$
is a non-negative definite kernel on $\mathbb B_r$, in particular on $\Omega$. Also, it follows from Corollary \ref{k-inv-cnp} that $K$ is CNP. 

\begin{remark} \label{remME3}

In~fact, \emph{all} CNP kernels on $\Omega$ arise essentially by a construction as in the last example. Namely, let $c_{\underline s}\ge0$ be such that
$$ \sum_{|\underline s|>0} c_{\underline s} K_{\underline s}(\mathbf e,\mathbf e) \le 1 $$
where $\mathbf e$ is any element of the Shilov boundary of~$\Omega$ (a~maximal tripotent, see~\cite{Arazy}; the value of $K_{\underline s}(\mathbf e,\mathbf e)$ is independent of the choice of $\mathbf e$ and there is an explicit formula for it in terms of the signature $\underline s$ and the domain invariants $r$, $a$ and $b$, see \cite[Theorem 3.4]{FK}). Setting
$$ L := \sum_{|\underline s|>0} c_{\underline s} K_{\underline s} , $$
we then have $0\le L(\boldsymbol z,\boldsymbol z)<1$ for all $\boldsymbol z\in\Omega$. The Cauchy-Schwarz inequality $|K_{\underline s}(\boldsymbol z,\boldsymbol w)|^2\le K_{\underline s}(\boldsymbol z,\boldsymbol z)K_{\underline s}(\boldsymbol w,\boldsymbol w)$ thus shows that the series for $L$ converges on all of $\Omega\times\Omega$, is bounded by 1 in modulus there, defines a nonnegative definite kernel. By Remark \ref{remME2},
$$ K := \frac1{1-L}  $$ will therefore be a CNP kernel on $\Omega$. Reversing this procedure also shows that all CNP kernels on $\Omega$ arise in this way.
\end{remark}

%Since $b_{\epsilon_1} = -\hat{b}_{\epsilon_1} \neq 0$, it follows from Lemma \ref{}, that the multiplication operators by the coordinates functions are bounded on the Hilbert space $(\mathcal{H}, K).$ 

%\begin{lemma}
%Let $K : \Omega \times \Omega \to \mathbb C$ be a $\mathbb K$-invariant CNP kernel and $\boldsymbol{M} = (M_1, \ldots, M_d)$ be the tuple of multiplication by the coordinate functions on $(\mathcal{H}, K)$. Then $\boldsymbol{M}$ is a $1/K$-contraction.    
%\end{lemma}

\section{CNP property and Characteristic function}

A non-vanishing unitary invariant kernel $K$ on $\mathbb B_d$ is said to be an admissible kernel if the tuple $\boldsymbol{M} = (M_{z_1}, \ldots, M_{z_d})$ of multiplication operators by the co-ordinate functions is bounded and $\boldsymbol{M}$ is a $\frac{1}{K}$-contraction. In \cite{Tiret}, it is proved that an admissible kernel $K$ on $\mathbb B_d$ possesses the CNP property if and only if every pure $\frac{1}{K}$-contraction admits a characteristic function. In this section, given a $\mathbb K$-invariant kernel on $\Omega$, we provide a suitable definition of $\frac{1}{K}$-contractivity. Our main objective of this section is to provide a necessary and sufficient condition for $\mathbb K$-invariant kernels to possess the CNP property in terms of the existence of the characteristic function of every pure $\frac{1}{K}$-contraction. 

Recall that if $K = \sum_{\s \geq \underline 0} a_{\s} K_{\s}$ is a $\mathbb K$-invariant kernel, then, due to Proposition \ref{pp2}, there exists a sequence of real numbers $\{b_{\s}\}_{\s > \underline 0}$ such that $1 - \frac{1}{K} = \sum_{\s > \underline 0} b_{\s}K_{\s}$. Unless otherwise mentioned, henceforth, for a $\mathbb K$-invariant kernel $K = \sum_{\s} a_{\s} K_{\s}$, $\{b_{\s}\}_{\s > \underline 0}$ represents a sequence of real numbers such that $1 - \frac{1}{K} = \sum_{\s > \underline 0} b_{\s}K_{\s}$. We begin with a few necessary definitions.

\begin{definition}
(i) Let $K = \sum_{\s} a_{\s}K_{\s}$ be a non-vanishing $\mathbb K$- invariant kernel. Suppose $\boldsymbol{T}$ is a commuting $d$-tuple of bounded operators. Assume that the series 
$$\displaystyle \sum_{\s > \underline 0} b_{\s} \sum_{\alpha=1}^{d_{\s}} \psi_\alpha^{\s}(\boldsymbol{T}) \psi_\alpha^{\s}(\boldsymbol{T})^*$$
converges strongly. If 
$$I - \displaystyle \sum_{\s > \underline 0} b_{\s} \sum_{\alpha=1}^{d_{\s}} \psi_\alpha^{\s}(\boldsymbol{T}) \psi_\alpha^{\s}(\boldsymbol{T})^* \geq 0,$$
then $\boldsymbol{T}$ is called a $\frac{1}{K}$-contraction.
Denote $$\Delta_{\boldsymbol T} := \left(I- \sum_{\s>\underline 0} b_{\s} \sum_{\alpha=1}^{d_{\s}} \psi_\alpha^{\s}(\boldsymbol T)\psi_\alpha^{\s}(\boldsymbol T)^*\right)^{\frac{1}{2}},$$
the positive square root of $I- \sum_{\s>\underline 0} b_{\s} \sum_{\alpha=1}^{d_{\s}} \psi_\alpha^{\s}(\boldsymbol T)\psi_\alpha^{\s}(\boldsymbol T)^*$.

(ii) A $\frac{1}{K}$-contraction $\boldsymbol{T}$ is said to be pure if $\displaystyle \sum_{\s \geq \underline 0} a_{\s} \sum_{\alpha=1}^{d_{\s}} \psi_{\alpha}^{\s}(\boldsymbol{T}) \Delta_{\boldsymbol T}^2 \psi_{\alpha}^{\s}(\boldsymbol{T})^*$ converges to the identity operator $I$ strongly.    
\end{definition}
\begin{definition}
A  $\mathbb K$-invariant kernel $K$ is called admissible if the operators of multiplication by the co-ordinate functions $M_{z_i}$ are bounded operators on $H_K$ for $i =1,\ldots, d$ and the $d$-tuple $\boldsymbol M =(M_{z_1},\ldots, M_{z_d})$ is a $\frac{1}{K}$-contraction.  
\end{definition}

In the following lemma, we provide a sufficient condition on a $\mathbb K$-invariant CNP kernel such that the multiplication operators by the co-ordinate functions on the corresponding Hilbert space are bounded.

\begin{lemma}\label{bounded}
Let $K$ be a $\mathbb K$-invariant CNP kernel and $H_K$ be the corresponding Hilbert space. Then, the multiplication operators $M_{z_i}$, $1 \leq i \leq d$, by the coordinate functions on $H_K$ are bounded. Furthermore, $K$ is an admissible kernel.
\end{lemma}

% \begin{lemma}\label{admissible}
% Let $\mathcal H$ be a $\mathbb K$-invariant CNP Space on a Cartan domain $\Omega$ with kernel $K$. Further assume that $b_{\epsilon_1}$ in \eqref{eqn:1-1/k} is non-zero. Then $K$ is an admissible kernel.
% \end{lemma}

% \begin{proof}
% It is enough to prove
%  $$ \Big\langle (I - \displaystyle \sum_{\s > \underline 0}^{|\s|=N} b_{\s} \sum_{\alpha=1}^{d_{\s}} \psi_\alpha^{\s}(\boldsymbol{M}) \psi_\alpha^{\s}(\boldsymbol{M})^*)K(., z) , K(., z)\Big\rangle \geq 0 ,$$ which is equivalent to prove
%  \beq \label{hartz5.2}
%  (I-\displaystyle \sum_{\s > \underline 0}^{|\s|=N}  b_{\s}\sum_{\alpha=1}^{d_{\s}} \psi_\alpha^{\s}(\boldsymbol{z}) \overline{\psi_\alpha^{\s}(\boldsymbol{z})})K(z,z)\geq 0.\eeq
%  Since $b_{\s}\geq 0$ for all $\s$ and
%  $$(I-\displaystyle \sum_{\s > \underline 0}  b_{\s}\sum_{\alpha=1}^{d_{\s}} \psi_\alpha^{\s}(\boldsymbol{z}) \overline{\psi_\alpha^{\s}(\boldsymbol{z})})K(z,z)=1,$$ Equation \eqref{hartz5.2} follows.
% \end{proof}
\begin{proof}
Recall that $a_{\s} > 0$ for each signature $\s \in \vec{N}_0^r$. From Equation \eqref{eqnx:sec1:kalu}, we have $b_{\epsilon_1} = -\hat{b}_{\epsilon_1} = a_{\epsilon_1}$ and therefore, $b_{\epsilon_1} > 0.$ Due to Corollary \ref{k-inv-cnp}, $b_{\s}\geq 0$ for all $\s$ and
 $$(1-\displaystyle \sum_{\s > \underline 0}  b_{\s}\sum_{\alpha=1}^{d_{\s}} \psi_\alpha^{\s}(\boldsymbol{z}) \overline{\psi_\alpha^{\s}(\boldsymbol{w})})K(\boldsymbol{z}, \boldsymbol{w})=1.$$ 
 Thus, it follows that 
  \beq 
 (1-\displaystyle \sum_{|\s| = 1}^{N}  b_{\s}\sum_{\alpha=1}^{d_{\s}} \psi_\alpha^{\s}(\boldsymbol{z}) \overline{\psi_\alpha^{\s}(\boldsymbol{w})})K(\boldsymbol{z}, \boldsymbol{w})\eeq
is non-negative definite. Taking $N = 1$ we see that, $(\frac{1}{b_{\epsilon_1}}-\langle\boldsymbol{z}, \boldsymbol{w}\rangle)K(\boldsymbol{z}, \boldsymbol{w})$
is non-negative definite. Hence by \cite[Lemma 3.1]{Bagchimisra} it follows that the multiplication operators $M_{z_i}$, $1 \leq i \leq d$, are bounded. 

The fact that the tuple of multiplication by the coordinate functions $\boldsymbol{M}$ on $H_K$ is a $\frac{1}{K}$-contraction follows by an argument similar to that in the proof of \cite[Lemma 5.2]{Hartz}.
\end{proof}

In the following, we provide a technical lemma.

\begin{lemma}\label{cha:existence fo b_s}
Let $K = \sum_{\s \geq \underline 0} a_{\s}K_{\s}$ be a $\mathbb K$-invariant kernel on $\Omega$. If $1 - \frac{1}{K} = \sum_{\s >0} b_{\s} K_{\s}$ for a sequence of real numbers $\{b_{\s}\}_{\s>\underline 0}$, then, for every signature $\underline p > 0$, we have
\begin{equation}\label{cha:eqn:1}
a_{\underline p} = \sum_{\substack{\s \geq \underline 0, \tilde{\s} > 0,\\ |\s| + |\tilde{\s}| = |\underline p|}} a_{\s} b_{\tilde{\s}} c_{\s, \tilde{\s}}^{\underline p}.  
\end{equation}
\end{lemma}

\begin{proof}
The existence of a sequence of real numbers $\{b_{\s}\}_{\s>\underline 0}$ such that $1 - \frac{1}{K} = \sum_{\s > \underline 0} b_{\s} K_{\s}$ follows from Proposition \ref{pp2}. The equation $1 - \frac{1}{K} = \sum_{\s > \underline 0} b_{\s} K_{\s}$ is equivalent to the following equation
$$\displaystyle \sum_{\s \geq \underline 0} a_{\s} K_{s} = 1 + \left(\sum_{\s \geq \underline 0} a_{\s} K_{s}\right) \left(\sum_{\tilde{\s} > \underline 0} b_{\tilde{\s}} K_{\tilde{s}}\right).$$
Using Equation \eqref{exist:eqn1 b_s} and then, comparing the coefficient of $K_{\underline p}$ from both sides of the above equation, we obtain  
$$a_{\underline p} = \sum_{\substack{\s \geq \underline 0, \tilde{\s} > \underline 0,\\ |\s| + |\tilde{\s}| = |\underline p|}}  a_{\s} b_{\tilde{\s}} c_{\s, \tilde{\s}}^{\underline p},$$
for every signature $\underline p > \underline 0.$
\end{proof}

The following proposition proves that the series $\sum_{\s \geq \underline 0} a_{\s} \sum_{\alpha=1}^{d_{\s}} \psi_{\alpha}^{\s}(\boldsymbol{T}) \Delta_{\boldsymbol{T}}^2 \psi_{\alpha}^{\s}(\boldsymbol{T})^*$
is always convergent for every $\frac{1}{K}$-contraction $\boldsymbol{T}$.

\begin{proposition}\label{lem:cha:2}
For any $\frac{1}{K}$-contraction $\boldsymbol{T} = (T_1, \ldots, T_d)$ the series
$$\displaystyle \sum_{\s \geq \underline 0} a_{\s} \sum_{\alpha=1}^{d_{\s}} \psi_{\alpha}^{\s}(\boldsymbol{T}) \Delta_{\boldsymbol{T}}^2 \psi_{\alpha}^{\s}(\boldsymbol{T})^*$$
converges strongly to a positive contraction.
\end{proposition}

\begin{proof}
If the convergence of the series, given in the statement of the lemma, is proved, then the positivity follows immediately. For $N \geq 1$, let 
$$S_N = \displaystyle \sum_{|\s| = 0}^N a_{\s} \sum_{\alpha = 1}^{d_{\s}} \psi_{\alpha}^{\s}(\boldsymbol{T}) \Delta_{\boldsymbol{T}}^2 \psi_{\alpha}^{\s}(\boldsymbol{T})^*.$$
Let $h \in \mathcal{H}$. Then,
\begin{flalign}\label{ch:eqn:new2}
\nonumber &\left\langle S_Nh, h \right\rangle\\
\nonumber &\phantom{xx}=  \displaystyle \sum_{|\s| = 0}^N a_{\s} \sum_{\alpha = 1}^{d_{\s}} \left\langle \psi_{\alpha}^{\s}(\boldsymbol{T}) \Delta_{\boldsymbol{T}}^2 \psi_{\alpha}^{\s}(\boldsymbol{T})^*h, h \right\rangle\\
 &\phantom{xx}= \displaystyle \sum_{|\s| = 0}^N a_{\s} \sum_{\alpha = 1}^{d_{\s}} \|\psi_{\alpha}^{\s}(\boldsymbol{T})^* h\|^2 - \displaystyle \sum_{|\s| = 0}^N a_{\s} \sum_{\tilde{\s} > \underline 0} \left\langle \sum_{\alpha = 1}^{d_{\s}} \sum_{\tilde{\alpha} = 1}^{d_{\tilde{\s}}} b_{\tilde{\s}} \psi_{\alpha}^{\s}(\boldsymbol{T}) \psi_{\tilde{\alpha}}^{\tilde{\s}}(\boldsymbol{T}) \psi_{\tilde{\alpha}}^{\tilde{\s}}(\boldsymbol{T})^*\psi_{\alpha}^{\s}(\boldsymbol{T})^* h, h \right\rangle.  
\end{flalign}
Note that the identity
$$\displaystyle \sum_{\alpha = 1}^{d_{\s}} \sum_{\tilde{\alpha} = 1}^{d_{\tilde{\s}}} \psi_{\alpha}^{\s}(\boldsymbol{T}) \psi_{\tilde{\alpha}}^{\tilde{\s}}(\boldsymbol{T}) \psi_{\tilde{\alpha}}^{\tilde{\s}}(\boldsymbol{T})^*\psi_{\alpha}^{\s}(\boldsymbol{T})^* = \displaystyle \sum_{|\underline p| = |\s| + |\tilde{\s}|} c^{\underline p}_{\s, \tilde{\s}} \sum_{\beta}^{d_{\underline p}} \psi_{\beta}^{\underline p}(\boldsymbol{T}) \psi_{\beta}^{\underline p}(\boldsymbol{T})^*$$
follows from Equation \eqref{exist:eqn1 b_s}. Using the above identity in Equation \eqref{ch:eqn:new2}, we obtain
\begin{flalign*}\label{ch:eqn:new2}
\left\langle S_Nh, h \right\rangle 
 &= \displaystyle \sum_{|\s| = 0}^N a_{\s} \sum_{\alpha = 1}^{d_{\s}} \|\psi_{\alpha}^{\s}(\boldsymbol{T})^* h\|^2 - \displaystyle \sum_{|\s| = 0}^N \sum_{\tilde{\s} > \underline 0} \sum_{|\underline p| = |\s| + |\tilde{\s}|} a_{\s} b_{\tilde{\s}} c^{\underline p}_{\s, \tilde{\s}}  \sum_{\beta = 1}^{d_{\underline p}} \|\psi_{\beta}^{\underline p}(\boldsymbol{T})^* h\|^2\\
 &\leq \displaystyle \sum_{|\s| = 0}^N a_{\s} \sum_{\alpha = 1}^{d_{\s}} \|\psi_{\alpha}^{\s}(\boldsymbol{T})^* h\|^2 - \displaystyle \sum_{1 \leq |\underline p| \leq N} \left( \sum_{\substack{0 \leq |\s| \leq N\\ \tilde{\s} > \underline 0\\ |\s| + |\tilde{\s}| = |\underline p|}} a_{\s} b_{\tilde{\s}} c^{\underline p}_{\s, \tilde{\s}} \right)  \sum_{\beta = 1}^{d_{\underline p}} \|\psi_{\beta}^{\underline p}(\boldsymbol{T})^* h\|^2\\
 &= \displaystyle \sum_{|\s| = 0}^N a_{\s} \sum_{\alpha = 1}^{d_{\s}} \|\psi_{\alpha}^{\s}(\boldsymbol{T})^* h\|^2 - \displaystyle \sum_{1 \leq |\underline p| \leq  N} a_{\underline p} \sum_{\beta = 1}^{d_{\underline p}} \|\psi_{\beta}^{\underline p}(\boldsymbol{T})^* h\|^2\\
 &= \|h\|^2.
\end{flalign*}
Here, the second last equality follows from Lemma \ref{cha:existence fo b_s}. This proves that the sequence of operators $\{S_N\}$ converges strongly to a contraction.
\end{proof}

\begin{remark}\label{rem 1}
For any $\frac{1}{K}$-contraction $\boldsymbol{T} = (T_1, \ldots, T_d)$, the proof of Proposition \ref{lem:cha:2} implies that 
$$S_N = \displaystyle \sum_{|\s| = 0}^N a_{\s} \sum_{\alpha = 1}^{d_{\s}} \psi_{\alpha}^{\s}(\boldsymbol{T}) \Delta_{\boldsymbol{T}}^2 \psi_{\alpha}^{\s}(\boldsymbol{T})^*$$
is a contraction for every $N \geq 1$ and therefore $\{\|S_N\|\}$ is bounded above by $1$.       
\end{remark}

The operator $V_T$, defined in the following corollary, plays a significant role in the theory $\frac{1}{K}$-calculus and also, in the theory of $\frac{1}{K}$-contractions (cf. \cite{Arazy2}). Following corollary is an immediate consequence of Proposition \ref{lem:cha:2} and \cite[Theorem 1.3]{Arazy2}.

\begin{corollary}\label{cor:vt contraction}
If $\boldsymbol{T} = (T_1, \ldots, T_d)$ is a $\frac{1}{K}$-contraction on a Hilbert space $\mathcal{H}$, then the operator $V_{\boldsymbol{T}} : \mathcal{H} \to H_{K} \otimes \overline{\mbox{Ran}}\,\, \Delta_{\boldsymbol{T}}$, defined by
\begin{equation}\label{char:eqn:def V_T}
    V_{\boldsymbol{T}}(h) = \displaystyle \sum_{\s \geq \underline 0} a_{\s} \sum_{\alpha = 1}^{d_{\s}} \psi_{\alpha}^{\s} \otimes \Delta_{\boldsymbol{T}} \psi_{\alpha}^{\s}(\boldsymbol{T})^*h,\,\,h \in \mathcal{H},
\end{equation}
is a contraction and satisfies
$$V_{\boldsymbol{T}}^* \left(p(\boldsymbol{M}_{\boldsymbol{z}} \otimes I_{\overline{\mbox{Ran}}\,\, \Delta_{\boldsymbol{T}}} )\right) = p(\boldsymbol{T})V_{\boldsymbol{T}}^*,$$
for every polynomials $p$.
\end{corollary}

Let $E_0$ denote the orthogonal projection of the reproducing kernel Hilbert space $H_K$ onto the one dimensional subspace consisting of the constant functions. The following lemma asserts that the $d$-tuple of multiplication operators by the co-ordinate functions $\boldsymbol{M} = (M_{z_{1}}, \ldots, M_{z_{d}})$ on $H_K$ is pure whenever $K$ is an admissible kernel.

\begin{lemma}\label{lem2.1}
Suppose $K$ is an admissible kernel. Then $\Delta_{\boldsymbol{M}} = E_0$ and the $d$-tuple of multiplication operators by the co-ordinate functions $\boldsymbol{M} = (M_{z_{1}}, \ldots, M_{z_{d}})$, acting on the Hilbert space $H_K$, is pure.
\end{lemma}

\begin{proof}

Let $H_{K}^0 := \left\lbrace \sum_{i=1}^n c_i K_{\boldsymbol{w}_i} : n \geq 1,\,\,c_1,\ldots, c_n \in \mathbb C,\,\,\boldsymbol{w}_1, \ldots, \boldsymbol{w}_n \in \Omega\right\rbrace.$ The space $H_K^0$ is dense in $H_K$. Take an arbitrary element $\sum_{i=1}^n c_i K_{\boldsymbol{w}_i}$ of $H_K^0$. A straightforward calculation implies that 
$$\Delta_{\boldsymbol{M}}^2 \left(\sum_{i=1}^n c_i K_{\boldsymbol{w}_i}\right) = \left\langle \sum_{i=1}^n c_i K_{\boldsymbol{w}_i}, K_0\right\rangle K_0.$$
Since $H_K^0$ is a dense subset of $H$, it follows that $\Delta_{\boldsymbol{M}}$ is the projection $E_0.$

For $N \geq 1$, let
$S_N$ denote $\sum_{|\s| \leq N} a_{\s} \sum_{\alpha=1}^{d_{\s}} \psi_{\alpha}^{\s}(\boldsymbol{M}) \Delta_{\boldsymbol{M}}^2 \psi_{\alpha}^{\s}(\boldsymbol{M})^*$. Suppose $\boldsymbol{w} \in \Omega$ be an arbitrary element. Since the sequence of partial sums $\{\sum_{|\s| \leq N} a_{\s} \sum_{\alpha=1}^{d_{\s}} \overline{\psi_{\alpha}^{\s}(\boldsymbol{w})} \psi_{\alpha}^{\s}(\cdot)\}_{N \geq 1}$ converges to $K_{\boldsymbol{w}}$ in the norm topology, given an $\epsilon > 0$, there exists a natural number $m$ such that 
\begin{equation}\label{ch1:eqn:ker-converge}
 \|K_{\boldsymbol{w}} - \sum_{|\s| \leq N} a_{\s} \sum_{\alpha=1}^{d_{\s}} \overline{\psi_{\alpha}^{\s}(\boldsymbol{w})} \psi_{\alpha}^{\s}(\cdot)\| < \epsilon   
\end{equation}
holds for every $N \geq m$. A straightforward computation implies that 
\begin{equation}\label{ch1:eqn:lem-pure}
S_N K_{\boldsymbol{w}} = \sum_{|\s| \leq N} a_{\s} \sum_{\alpha=1}^{d_{\s}} \overline{\psi_{\alpha}^{\s}(\boldsymbol{w})} \psi_{\alpha}^{\s}(\cdot).    
\end{equation}
Combining Equations \eqref{ch1:eqn:ker-converge} and \eqref{ch1:eqn:lem-pure}, we obtain 
$$\|K_{\boldsymbol{w}} - S_{N} K_{\boldsymbol{w}}\| < \epsilon$$
holds for every $N \geq m$. This implies that the sequence $\{S_NK_{\boldsymbol{w}}\}$ converges to $K_{\boldsymbol{w}}$. Consequently, for each $h$ in $H_K^0$, the sequence $\{S_N h\}$ is convergent. Combining the facts that $H_K^0$ is a dense subset of $H_K$ and that the sequence $\{\|S_N\|\}$ is bounded - a consequence of Proposition \ref{lem:cha:2} - we obtain that $\{S_N\}$ converges to the identity operator in the strong operator topology of $H_K$.
\end{proof}

 The following theorem is due to Arazy and Engli\v{s}. For the proof of the following theorem, see \cite[Theorem 1.3]{Arazy2}.

\begin{theorem}\label{isometry}
Let $\boldsymbol{T} = (T_1,\ldots, T_d)$ be a commuting tuple of operators on a Hilbert space $\mathcal{H}$. If $\boldsymbol{T}$ is a pure $\frac{1}{K}$-contraction, then the map $V_T:\mathcal H\to H_K\otimes \overline{\text{Ran}}\,\,\Delta_{\boldsymbol{T}}$, given by
$$ h \to  \displaystyle \sum_{\s > \underline 0} \sum_{\alpha = 1}^{d_{\s}} a_{\s} \psi_{\alpha}^{\s} \otimes \Delta_{\boldsymbol{T}} \left(\psi^{\s}_{\alpha} (\boldsymbol{T})\right)^*h,$$
is an isometry satisfying
$$V_{\boldsymbol{T}}^* \left(p(\boldsymbol{M}) \otimes I_{\overline{\mbox{Ran}}\,\,\Delta_{\boldsymbol{T}}}\right) = p(\boldsymbol{T}) V_{\boldsymbol{T}}^*$$
for all $p\in \mathbb C[z_1, \ldots, z_d]$.
\end{theorem} 

We recall the definition of a $(K, \boldsymbol{ T})$-factorable positive operator from \cite[Definition 2.3]{Tiret}. Suppose $K$ is an admissible kernel and $\boldsymbol{T} = (T_1, \ldots, T_d)$ is a commuting $d$-tuple of operators on $\mathcal{H}$. A positive operator $X \in \mathcal{B}(\mathcal{H})$ is said to be $(K, \boldsymbol{T})$-factorable if $X$ has closed range and there exists a Hilbert space $\mathcal{L}$ with a bounded linear transformation $\Theta : H_{K} \otimes \mathcal{L} \to \mathcal{H}$ such that $X = \Theta \Theta^*$ and $\Theta \left(M_{z_i} \otimes I_{\mathcal{L}} \right) = T_i \Theta$ holds for each $i$. The following proposition is the key ingredient of the main theorem of this section. Both the proposition and proof are motivated from \cite[Proposition 2.4]{Tiret}.

\begin{proposition}\label{prop 4.2}
  $X \in \mathcal{B}(\mathcal{H})$ is $(K, \boldsymbol{T})$-factorable if and only if
  \begin{enumerate}
      \item for all $i, \|M_{z_i}\|^2X- T_iXT_i^*\geq 0.$
      \item $P_{\boldsymbol T}(X)=\displaystyle \sum_{\s > \underline 0} b_{\s} \sum_{\alpha=1}^{d_{\s}} \psi_\alpha^{\s}(\boldsymbol{T}) X \psi_\alpha^{\s}(\boldsymbol{T})^*$ converges strongly such that $X-P_{\boldsymbol T}(X)\geq 0.$
      \item $\displaystyle \sum_{\s \geq \underline 0} a_{\s} \sum_{\alpha=1}^{d_{\s}} \psi_\alpha^{\s}(\boldsymbol{T}) \left (X- P_{\boldsymbol T}(X)\right) \psi_\alpha^{\s}(\boldsymbol{T})^*$ converges to $X.$
  \end{enumerate}
\end{proposition}
\begin{proof}
  Let $X$ be $(K, \boldsymbol{T})$-factorable, that is, there exists a Hilbert space $\mathcal{L}$ with a bounded linear transformation $\Theta : H_{K} \otimes \mathcal{L} \to \mathcal{H}$ such that $X = \Theta \Theta^*$ and $\Theta \left(M_{z_i} \otimes I_{\mathcal{L}} \right) = T_i \Theta$ holds for each $i$.   Then \begin{align*}
      \|M_{z_i}\|^2X- T_iXT_i^*&= \|M_{z_i}\|^2\Theta\Theta^*- T_i\Theta\Theta^*T_i^*\\
      &=\|M_{z_i}\|^2\Theta\Theta^*- \Theta\left(M_{z_i} \otimes I_{\mathcal{L}} \right)\left(M_{z_i}^* \otimes I_{\mathcal{L}} \right)\Theta^* \\
      &=\Theta\Big(\|M_{z_i}\|^2 I- \left(M_{z_i} \otimes I_{\mathcal{L}} \right)\left(M_{z_i}^* \otimes I_{\mathcal{L}} \right)\Big)\Theta^* \geq 0.
  \end{align*}
This proves $(1).$ The following computation implies (2):
  \begin{align*}
      X-P_{\boldsymbol T}(X)&=X-\displaystyle \sum_{\s > \underline 0} b_{\s} \sum_{\alpha=1}^{d_{\s}} \psi_\alpha^{\s}(\boldsymbol{T}) X \psi_\alpha^{\s}(\boldsymbol{T})^*\\
      &=\Theta\big(I-\displaystyle \sum_{\s > \underline 0} b_{\s} \sum_{\alpha=1}^{d_{\s}} (\psi_\alpha^{\s}(\boldsymbol M) \otimes I_{\mathcal{L}}) ( \psi_\alpha^{\s}(\boldsymbol M)^* \otimes I_{\mathcal{L}})\Big)\Theta^*\geq 0.
  \end{align*}
Also, observe that 
  
\begin{align*}
    &\displaystyle \sum_{|\s| = 0}^{N} a_{\s}\sum_{\alpha=1}^{d_{\s}} \psi_\alpha^{\s}(\boldsymbol{T}) (X- P_{\boldsymbol T}(X))\psi_\alpha^{\s}(\boldsymbol{T})^* \\
    &= \displaystyle \sum_{|\s| = 0}^{N} a_{\s} \sum_{\alpha=1}^{d_{\s}} \psi_\alpha^{\s}(\boldsymbol{T}) (\Theta\Theta^*- P_{\boldsymbol T}(\Theta\Theta^*))\psi_\alpha^{\s}(\boldsymbol{T})^* \\ 
    &= \Theta\left[\displaystyle \sum_{|\s| = 0}^{N} a_{\s} \sum_{\alpha=1}^{d_{\s}} (\psi_\alpha^{\s}(\boldsymbol M \otimes I_{\mathcal{L}}) \psi_\alpha^{\s}(\boldsymbol M \otimes I_{\mathcal{L}})^* - \displaystyle \sum_{|\s| = 0}^{N} \sum_{\alpha = 1}^{d_{\s}} a_{\s} \psi_\alpha^{\s}(\boldsymbol M \otimes I_{\mathcal{L}})\right.\\
    &\phantom{ga} \left. \left(\displaystyle \sum_{\tilde{\s}>\underline 0} b_{\tilde{\s}} \sum_{\alpha=1}^{d_{\tilde{\s}}} (\psi_{\alpha}^{\tilde{\s}}(\boldsymbol M \otimes I_{\mathcal{L}}) \psi_{\alpha}^{\tilde{\s}}(\boldsymbol M \otimes I_{\mathcal{L}})^* \right) \psi_{\alpha}^{\s} (\boldsymbol{M} \otimes I)^*\right] \Theta^*.
\end{align*}
Since $\boldsymbol{M}$ is a pure $\frac{1}{K}$-contraction, it follows that 
\begin{align*}
    &\left[\displaystyle \sum_{|\s| = 0}^{N} a_{\s} \sum_{\alpha=1}^{d_{\s}} \psi_\alpha^{\s}(\boldsymbol M \otimes I_{\mathcal{L}}) \psi_\alpha^{\s}(\boldsymbol M \otimes I_{\mathcal{L}})^* - \displaystyle \sum_{|\s| = 0}^{N} \sum_{\alpha = 1}^{d_{\s}} a_{\s} \psi_\alpha^{\s}(\boldsymbol M \otimes I_{\mathcal{L}})\right.\\
    &\phantom{ga} \left. \left(\displaystyle \sum_{\tilde{\s}>\underline 0} b_{\tilde{\s}} \sum_{\alpha=1}^{d_{\tilde{\s}}} \psi_\alpha^{\tilde{\s}}(\boldsymbol M \otimes I_{\mathcal{L}}) \psi_\alpha^{\tilde{\s}}(\boldsymbol M \otimes I_{\mathcal{L}})^* \right) \psi_{\alpha}^{\s} (\boldsymbol{M} \otimes I)^*\right]
\end{align*}
converges to $I_{H_K \otimes \mathcal{L}}$ in the strong operator topology and therefore, $(3)$ holds.

  Conversely, let us assume $(1), (2)$, and $(3).$
For each $i=1, \ldots, d,$ define the operator $A_i$ on the range of $X$, $\mbox{Ran}\,X$, by
  $A_i X^{1/2}h = X^{1/2}T_i^*h,$ $h\in \mathcal H.$ It is easy to see that $\boldsymbol{A}=(A_1, \ldots, A_d)$ is a commuting tuple of bounded linear operators.
  Now, using the condition $(2)$, we get
  \begin{align*}
      &\Big\langle  \displaystyle \sum_{\s > \underline 0} b_{\s} \sum_{\alpha=1}^{d_{\s}} \psi_\alpha^{\s}(\boldsymbol{A}^*) \psi_\alpha^{\s}(\boldsymbol{A})X^{1/2}h, X^{1/2}h\Big\rangle 
      \\&\phantom{xx}=\Big\langle  \displaystyle \sum_{\s > \underline 0} b_{\s} \sum_{\alpha=1}^{d_{\s}} \psi_\alpha^{\s}(\boldsymbol{T}^*)X \psi_\alpha^{\s}(\boldsymbol{T})h, h\Big\rangle \leq \Big\langle X h, h\Big\rangle = \Big\langle X^{1/2} h, X^{1/2}h\Big\rangle.
  \end{align*}
  Thus $\boldsymbol{A}^*$ is a $\frac{1}{K}$-contraction on $\text{Ran}\,\,X.$
  Also, a similar argument with condition $(3)$ shows that $\boldsymbol{A}^*$ is pure:
  \begin{align*}
      &\Big\langle  \displaystyle \sum_{\s \geq \underline 0} a_{\s} \sum_{\alpha=1}^{d_{\s}} \psi_\alpha^{\s}(\boldsymbol{A}^*)\Delta_{\boldsymbol{A}^*}^2 \psi_\alpha^{\s}(\boldsymbol{A})X^{1/2}h, X^{1/2}h\Big\rangle\\&\phantom{xx}=
      \Big\langle  \displaystyle \sum_{\s \geq \underline 0} a_{\s} \sum_{\alpha=1}^{d_{\s}} \psi_\alpha^{\s}(\boldsymbol{T}^*)\left (X- P_{\boldsymbol T(X)}\right) \psi_\alpha^{\s}(\boldsymbol{T})h, h\Big\rangle = \Big\langle Xh, h\Big\rangle.
  \end{align*}
  Now Theorem \ref{isometry} yields there exists an isometry $V:\text{Ran}\,\,X\to H_K\otimes \overline{\text{Ran}}\,\,\Delta_{\boldsymbol A}$ such that $$V^* \left(M_i\otimes I_{\overline{\mbox{Ran}}\,\, \Delta_{\boldsymbol A}}\right) = A_i V^*.$$
  Finally, if we define $\Theta=X^{1/2}V^*$ it is easy to check $X = \Theta \Theta^*$ and for each $i$ $$\Theta \left(M_i\otimes I_{\overline{\mbox{Ran}}\,\,\Delta_{\boldsymbol A}}\right) = A_i \Theta.$$  \end{proof}

Below, we provide the definition of the characteristic function of a $\frac{1}{K}$-contraction, following \cite[Definition 2.5]{Tiret}. The definition of the characteristic function of a $\frac{1}{K}$-contraction $\boldsymbol{T}$ in \cite[Definition 2.5]{Tiret} is given by assuming that $\boldsymbol{T}$ is pure. In our context, we give the definition of the characteristic function for any $\frac{1}{K}$-contraction.

\begin{definition}\label{def:existence of cha fun}
Let $\boldsymbol{T} = (T_1, \ldots, T_d)$ be a $\frac{1}{K}$-contraction. The tuple $\boldsymbol{T}$ is said to admit a characteristic function if there exists a Hilbert space $\mathcal{E}$ and an analytic function $\Theta_{\boldsymbol{T}} : \Omega \to \mathcal{B}(\mathcal{E}, \overline{\mbox{Ran}}\,\, \Delta_{\boldsymbol{T}})$ such that the corresponding multiplication operator $M_{\Theta_{\boldsymbol{T}}} : H_{K} \otimes \mathcal{E} \to H_{K} \otimes \overline{\mbox{Ran}}\,\, \Delta_{\boldsymbol{T}}$ fulfills the following identity
$$I - V_{\boldsymbol{T}} V_{\boldsymbol{T}}^* = M_{\Theta_{\boldsymbol{T}}} M_{\Theta_{\boldsymbol{T}}}^*.$$
\end{definition}

If $K$ is an admissible kernel, it follows from Theorem \ref{isometry} that $\ker V_{\boldsymbol{T}}^*$ is an invariant subspace for each $M_{z_i} \otimes I_{\overline{\mbox{\tiny{Ran}}}\,\,\Delta_{\boldsymbol{T}}}$, $1 \leq i \leq d$.

\begin{definition}\label{Def-BT}
    For a pure $\frac{1}{K}$-contraction $\boldsymbol T= (T_1, \ldots, T_d),$ associated $d$-tuple of commuting operators $B_{\boldsymbol T}$ is defined as
\begin{equation}\label{Eqn:Def-BT}
B_{\boldsymbol T}:=\Big((M_{z_1}\otimes I_{\overline{\mbox{\tiny{Ran}}}\,\,\Delta_{\boldsymbol{T}}})|_{\ker V_{\boldsymbol T}^*}, \ldots, (M_{z_d}\otimes I_{\overline{\mbox{\tiny{Ran}}}\,\,\Delta_{\boldsymbol{T}}})|_{\ker V_{\boldsymbol T}^*}\Big).    
\end{equation}    
\end{definition}

The theorem below is one of the important theorems of this section. It gives a necessary and sufficient criterion for the existence of the characteristic function of a pure $\frac{1}{K}$-contraction.

\begin{theorem}\label{B_T}
    Let $K$ be an admissible $\mathbb K$-invariant kernel over the Cartan domain $\Omega.$ A pure $\frac{1}{K}$-contraction $\boldsymbol T= (T_1, \ldots, T_d)$ on a Hilbert Space $\mathcal H$ admits a characteristic function if and only if $B_{\boldsymbol T}$ is a $\frac{1}{K}$-contraction.
\end{theorem}
\begin{proof}
  Let $\boldsymbol T$ admit a characteristic function. Then there exists a $B(\mathcal L, \overline{\text{Ran}}\,\,\Delta_{\boldsymbol T})-$ valued bounded analytic function $\theta_{\boldsymbol T}$ on $\Omega$ such that
  $$I-V_{\boldsymbol T}V_{\boldsymbol T}^*=M_{\theta_{\boldsymbol T}}M_{\theta_{\boldsymbol T}}^*.$$ Thus for any $N\geq 1,$
  \begin{align}\label{Main2}
    \nonumber& (I-V_{\boldsymbol T}V_{\boldsymbol T}^*)-\Big(\displaystyle \sum_{|\s| = 1}^{N} b_{\s} \sum_{\alpha=1}^{d_{\s}} (\psi_\alpha^{\s}(\boldsymbol{M})\otimes I)(I-V_{\boldsymbol T}V_{\boldsymbol T}^*) (\psi_\alpha^{\s}(\boldsymbol{M})^*\otimes I) \Big)\\
    &\phantom{xxxx}= M_{\theta_{\boldsymbol T}}\Big(I-\displaystyle \sum_{|\s| = 1}^{N} b_{\s} \sum_{\alpha=1}^{d_{\s}} (\psi_\alpha^{\s}(\boldsymbol{M})\otimes I) (\psi_\alpha^{\s}(\boldsymbol{M})^*\otimes I) \Big)M_{\theta_{\boldsymbol T}}^*.
  \end{align}
  Since $K$ is an admissible kernel, it follows that the tuple of multiplication operators $\boldsymbol M$ is $\frac{1}{K}$-contraction and therefore, the series $\displaystyle \sum_{\s > \underline 0} b_{\s} \sum_{\alpha=1}^{d_{\s}} (\psi_\alpha^{\s}(\boldsymbol{M})\otimes I) (\psi_\alpha^{\s}(\boldsymbol{M})^*\otimes I) \Big)$ converges strongly and $\Big(I-\displaystyle \sum_{\s > \underline 0} b_{\s} \sum_{\alpha=1}^{d_{\s}} (\psi_\alpha^{\s}(\boldsymbol{M})\otimes I) (\psi_\alpha^{\s}(\boldsymbol{M})^*\otimes I) \Big)\geq 0.$ This with \eqref{Main2} implies that $B_{\boldsymbol{T}}$ is a $\frac{1}{K}$-contraction.

Conversely, assume that $B_{\boldsymbol{T}}$ is a $\frac{1}{K}$-contraction. Let $c_i=\|M_{z_i}\|^2.$ Suppose $P_{\ker V_{\boldsymbol{T}}^*} := I - V_{\boldsymbol{T}}V_{\boldsymbol{T}}^*$ is the projection of $H_K\otimes \overline{\text{Ran}}\,\,\Delta_{\boldsymbol{T}}$ onto $\ker V_{\boldsymbol{T}}^*.$ For each $i, $ define a linear operator $R_i:=M_{z_i}\otimes I|_{\ker V_{\boldsymbol{T}}^*}.$ Then $\|R_i\|^2\leq c_i.$ In other words, $c_iI\geq R_i R_i^*,$ i.e.,
$$c_i(I-V_{\boldsymbol T}V_{\boldsymbol T}^*)-(M_{z_i}\otimes I)(I-V_{\boldsymbol T}V_{\boldsymbol T}^*)(M_{z_i}\otimes I)^*\geq 0.$$
This proves condition $(1)$ of Proposition \ref{prop 4.2}.
Since $\boldsymbol{M}$ is a $\frac{1}{K}$-contraction, it follows that the series 
\begin{align*}
  &\displaystyle \sum_{\s > \underline 0} b_{\s} \sum_{\alpha=1}^{d_{\s}} \psi_\alpha^{\s}(\boldsymbol{M}\otimes I|_{\ker V_T^*}) \psi_\alpha^{\s}(\boldsymbol{M}\otimes I|_{\ker V_T^*})^*\\&
 \phantom{xxxxxxx} =\displaystyle \sum_{\s > \underline 0} b_{\s} \sum_{\alpha=1}^{d_{\s}} \psi_\alpha^{\s}(\boldsymbol{M}\otimes I)(I-V_{\boldsymbol T}V_{\boldsymbol T}^*) \psi_\alpha^{\s}(\boldsymbol{M}\otimes I)^*  
\end{align*}
converges strongly and
\begin{align*}
    &(I-V_{\boldsymbol T}V_{\boldsymbol T}^*)-\displaystyle \sum_{\s > \underline 0} b_{\s} \sum_{\alpha=1}^{d_{\s}} \psi_\alpha^{\s}(\boldsymbol{M}\otimes I)(I-V_{\boldsymbol T}V_{\boldsymbol T}^*) \psi_\alpha^{\s}(\boldsymbol{M}\otimes I)^*\\
    &\phantom{xxxxxxxxxxxxxxxxxxxx}=(I-V_{\boldsymbol T}V_{\boldsymbol T}^*)-P_{\boldsymbol{M} \otimes I}(I-V_{\boldsymbol T}V_{\boldsymbol T}^*)\geq 0,
\end{align*}
where the map $P_{\boldsymbol{M} \otimes I}$ is defined in Proposition \ref{prop 4.2} $(2)$. This proves condition $(2)$ of Proposition \ref{prop 4.2}.

Let $N\geq 1$ be a natural number. Suppose
$$S_N=\displaystyle \sum_{|\s| = 0}^{N} a_{\s} \sum_{\alpha=1}^{d_{\s}} \psi_\alpha^{\s}(\boldsymbol{M}\otimes I)\Big((I-V_{\boldsymbol T}V_{\boldsymbol T}^*)-P_{\boldsymbol{M} \otimes I}(I-V_{\boldsymbol T}V_{\boldsymbol T}^*)\Big) \psi_\alpha^{\s}(\boldsymbol{M}\otimes I)^*.$$
We prove that $S_N$ converges strongly to $(I-V_{\boldsymbol T}V_{\boldsymbol T}^*).$
Consider an element $\sum_{i=1}^n( K_{w_i}\otimes x_i)\in H_K\otimes \overline{\text{Ran}}\,\,\Delta_{\boldsymbol{T}}.$ Then

\begin{align*}
   &\Big\langle S_N\sum_{i=1}^n( K_{\boldsymbol w_i}\otimes x_i), \sum_{i=1}^n( K_{\boldsymbol w_i}\otimes x_i)\Big\rangle \\ 
    &=\sum_{i,j=1}^n\Big[\Big\langle \displaystyle \sum_{|\s| \leq N} a_{\s} \sum_{\alpha=1}^{d_{\s}} \psi_\alpha^{\s}(\boldsymbol{M}\otimes I)\Big((I-V_{\boldsymbol T}V_{\boldsymbol T}^*)\Big) \psi_\alpha^{\s}(\boldsymbol{M}\otimes I)^*K_{\boldsymbol w_i}\otimes x_i, K_{\boldsymbol w_j}\otimes x_j\Big\rangle\\
    &-\Big\langle \displaystyle \sum_{|\s| \leq N} a_{\s} \sum_{\alpha=1}^{d_{\s}} \psi_\alpha^{\s}(\boldsymbol{M}\otimes I)\Big(P_{\boldsymbol{M} \otimes I}(I-V_{\boldsymbol T}V_{\boldsymbol T}^*)\Big) \psi_\alpha^{\s}(\boldsymbol{M}\otimes I)^*K_{\boldsymbol w_i}\otimes x_i, K_{\boldsymbol w_j}\otimes x_j\Big\rangle\Big]\\    
    &=\sum_{i,j=1}^n\Big[\Big\langle \displaystyle \sum_{\s \leq N} a_{\s} \sum_{\alpha=1}^{d_{\s}}\overline{\psi_\alpha^{\s}(\boldsymbol w_i)} \psi_\alpha^{\s}(\boldsymbol{M}\otimes I)\Big((I-V_{\boldsymbol T}V_{\boldsymbol T}^*)\Big) K_{\boldsymbol w_i}\otimes x_i, K_{\boldsymbol w_j}\otimes x_j\Big\rangle\\
    &-\Big\langle \displaystyle \sum_{|\s| \leq N} a_{\s} \sum_{\alpha=1}^{d_{\s}}\overline{\psi_\alpha^{\s}(\boldsymbol w_i)}  \psi_\alpha^{\s}(\boldsymbol{M}\otimes I)\sum_{\tilde{\s}\geq \underline 0} b_{\tilde{\s}} \sum_{\tilde{\alpha}=1}^{d_{\tilde{\s}}}\overline{\psi_{\tilde{\alpha}}^{\tilde{\s}}(\boldsymbol w_i)}\psi_{\tilde{\alpha}}^{\tilde{\s}}(\boldsymbol{M}\otimes I)(I-V_{\boldsymbol T}V_{\boldsymbol T}^*)\Big)K_{\boldsymbol w_j}\otimes x_j,\\
&\phantom{xxxxxxxxxxxxxxxxxxxxxxxxxxxxxxxxxxxxxxxxxxxxxxxxxxxxxxxxx}K_{\boldsymbol w_j}\otimes x_j \Big\rangle\Big]\\
    &=\sum_{i,j=1}^n\Big[\displaystyle \sum_{|\s| = 0}^{ N} a_{\s} \sum_{\alpha=1}^{d_{\s}}\overline{\psi_\alpha^{\s}(\boldsymbol w_i)} \psi_\alpha^{\s}(\boldsymbol w_j)\Big\langle \Big((I-V_{\boldsymbol T}V_{\boldsymbol T}^*)\Big) K_{\boldsymbol w_i}\otimes x_i, K_{\boldsymbol w_j}\otimes x_j\Big\rangle\\
    &-\displaystyle \sum_{|\s| = 0}^{ N} a_{\s} \sum_{\alpha=1}^{d_{\s}}\overline{\psi_\alpha^{\s}(\boldsymbol w_i)}  \psi_\alpha^{\s}(\boldsymbol w_j)\sum_{\tilde{\s}\geq \underline 0} b_{\tilde{\s}} \sum_{\tilde{\alpha}=1}^{d_{\tilde{\s}}}\overline{\psi_{\tilde{\alpha}}^{\tilde{\s}}(\boldsymbol w_i)}\psi_{\tilde{\alpha}}^{\tilde{\s}}(\boldsymbol w_j)\Big\langle (I-V_{\boldsymbol T}V_{\boldsymbol T}^*)\Big)K_{\boldsymbol w_i}\otimes x_i, K_{\boldsymbol w_j}\otimes x_j \Big\rangle\Big]
    \\
    &=\displaystyle \sum_{|\s| = 0}^{N} a_{\s} \sum_{\alpha=1}^{d_{\s}}\overline{\psi_\alpha^{\s}(\boldsymbol w_i)}  \psi_\alpha^{\s}(\boldsymbol w_j)\Big[1-\sum_{\tilde{\s}\geq 0} b_{\tilde{\s}} \sum_{\tilde{\alpha}=1}^{d_{\tilde{\s}}}\overline{\psi_{\tilde{\alpha}}^{\tilde{\s}}(\boldsymbol w_i)}\psi_{\tilde{\alpha}}^{\tilde{\s}}(\boldsymbol w_j)\Big]\Big\langle (I-V_{\boldsymbol T}V_{\boldsymbol T}^*)\Big)\sum_{i=1}^nK_{\boldsymbol w_i}\otimes x_i,\\
    &\phantom{xxxxxxxxxxxxxxxxxxxxxxxxxxxxxxxxxxxxxxxxxxxxxxxxxxxxxxx}\sum_{i=1}^n K_{\boldsymbol w_i}\otimes x_i \Big\rangle.
\end{align*}
The equation above implies that $S_{N} \leq (I-V_{\boldsymbol T}V_{\boldsymbol T}^*)$.
Also, note that the series 
$$\displaystyle \sum_{|\s| = 0}^{N} a_{\s} \sum_{\alpha=1}^{d_{\s}}\overline{\psi_\alpha^{\s}(\boldsymbol w_i)}  \psi_\alpha^{\s}(\boldsymbol w_j)\Big[1-\sum_{\tilde{\s}\geq \underline 0} b_{\tilde{\s}} \sum_{\tilde{\alpha}=1}^{d_{\tilde{\s}}}\overline{\psi_{\tilde{\alpha}}^{\tilde{\s}}(\boldsymbol w_i)}\psi_{\tilde{\alpha}}^{\tilde{\s}}(\boldsymbol w_j)\Big]$$
converges to 1 as $N\to \infty,$ which implies that $\{S_N\}$ converges to $(I-V_{\boldsymbol T}V_{\boldsymbol T}^*)$ for every element of the set $(H_K\otimes \overline{\text{Ran}}\,\,\Delta_{\boldsymbol{T}})_0 := \{\sum_{i=1}^n K_{\boldsymbol w_i}\otimes x_i: n\geq 1, \boldsymbol w_1, \ldots, \boldsymbol w_n\in \Omega, x_1,\ldots, x_n\in \overline{\text{Ran}}\,\,\Delta_{\boldsymbol{T}}\}$. Since  $(H_K\otimes \overline{\text{Ran}}\,\,\Delta_{\boldsymbol{T}})_0$ is a dense subset of $H_K\otimes \overline{\text{Ran}}\,\,\Delta_{\boldsymbol{T}}$ and $\{S_N\}$ is bounded above by $(I-V_{\boldsymbol T}V_{\boldsymbol T}^*)$, it follows that $S_N$ converges to $(I-V_{\boldsymbol T}V_{\boldsymbol T}^*)$ in strong operator topology. This proves condition $(3)$ of the Proposition \ref{prop 4.2}. Therefore, $(I-V_{\boldsymbol T}V_{\boldsymbol T}^*)$ is $(K, \boldsymbol{M}\otimes I)$-factorable. In other words, $\boldsymbol T$ admits a characteristic function.
\end{proof}

Recall from \cite{Tiret} that an admissible kernel $K$ is said to admit a characteristic function if every pure $\frac{1}{K}$-contraction admits a characteristic function. A sufficient condition for a $\mathbb K$-invariant kernel to have CNP property is obtained in the theorem below.
\begin{theorem}\label{Existence of cha fun}
   If $K$ is a $\mathbb K$-invariant CNP kernel, then $K$ admits a characteristic function.
\end{theorem}
\begin{proof}
  Let $\boldsymbol T$ be a pure $\frac{1}{K}$-contraction. To show that $\boldsymbol T$ admits a characteristic function, it is enough to show $B_{\boldsymbol T}$ is a $\frac{1}{K}$-contraction (see Theorem \ref{B_T}). Since $K$ is an admissible kernel,  $\boldsymbol{M}$ on $H_K$ is a $\frac{1}{K}$-contraction. Thus $\boldsymbol{M}\otimes I$ on $H_K \otimes \overline{\mbox{Ran}}\,\, \Delta_{\boldsymbol{T}}$ is also $\frac{1}{K}$-contraction.
  For any $N\geq 1,$
  \begin{align*}
      \displaystyle \sum_{|\s| = 1}^{N} b_{\s} \sum_{\alpha=1}^{d_{\s}} \psi_\alpha^{\s}(B_{\boldsymbol T}) \psi_\alpha^{\s}(B_{\boldsymbol T})^* &=\displaystyle \sum_{|\s| = 1}^{N} b_{\s} \sum_{\alpha=1}^{d_{\s}} P_{\ker V_{\boldsymbol T}^*}(\psi_\alpha^{\s}(\boldsymbol{M})\otimes I) P_{\ker V_{\boldsymbol T}^*}(\psi_\alpha^{\s}(\boldsymbol{M})^*\otimes I )|_{{\ker V_{\boldsymbol T}^*}}\\
      &\leq \displaystyle \sum_{\s > \underline 0} b_{\s} \sum_{\alpha=1}^{d_{\s}} P_{\ker V_{\boldsymbol T}^*}(\psi_\alpha^{\s}(\boldsymbol{M})\otimes I) (\psi_\alpha^{\s}(\boldsymbol{M})^*\otimes I )|_{{\ker V_{\boldsymbol T}^*}}\\
      &\leq I_{\ker V_{\boldsymbol T}^*}.
  \end{align*}
  Thus $B_{\boldsymbol T}$ is a $\frac{1}{K}$-contraction.
\end{proof}

Now, we have all the necessary ingredients to prove the main theorem of this section. It professes that the CNP property for an admissible kernel on a Cartan domain is characterized by the requirement of the existence of the characteristic function of every pure $\frac{1}{K}$-contraction. A similar characterization of a unitary invariant kernel on a unit ball is proved in \cite[Theorem 3.4]{Tiret}.  

\begin{theorem}\label{cnp main}
Suppose $K$ is an admissible kernel on a Cartan domain $\Omega$. Then, the kernel $K$ is CNP if and only if any pure $\frac{1}{K}$-contraction admits a characteristic function.
\end{theorem}

\begin{proof}
The proof of the forward direction follows from Theorem \ref{Existence of cha fun}. Therefore, we prove the converse direction. Assume that any pure $\frac{1}{K}$-contraction admits a characteristic function. Let $\mathcal{H}$ be Hilbert space. It is trivial to verify that the $d$-tuple $\boldsymbol{T} = (0, \ldots, 0)$ of zero operators on $\mathcal{H}$ is a pure $\frac{1}{K}$-contraction and $\Delta_{\boldsymbol{T}} = I_{\mathcal{H}}$. Consequently, we have $\mbox{Ran}\,\Delta_{\boldsymbol{T}} = \mathcal{H}$. From Theorem \ref{isometry}, it follows that the operator $V_{\boldsymbol{T}} : \mathcal{H} \to H_{K} \otimes \overline{\mbox{Ran}}\,\Delta_{\boldsymbol{T}}$ is given by 
$$V_{\boldsymbol{T}} (h) = 1 \otimes h,\,\,h \in \mathcal{H},$$
where $1$ denotes the constant function in $H_K$ which maps every element of $\Omega$ to $1$. 

Let $H_K^c$ denote the set of all constant functions in $H_K$ and $E_0^\perp$ denote the orthogonal projection of $H_K$ onto $(H_K^c)^\perp$. A direct computation verifies that $\ker (I - V_{\boldsymbol{T}}) = (H_K^c)^\perp \otimes \mathcal{H}$. Due to Theorem \ref{Existence of cha fun}, it follows that $\boldsymbol{B}_{\boldsymbol{T}}$, defined in Equation \eqref{Eqn:Def-BT}, is a $\frac{1}{K}$-contraction. A straightforward computation implies that $\boldsymbol{B}_{\boldsymbol{T}}$ is a $\frac{1}{K}$-contraction if and only if
\begin{equation}\label{main thm:eqn0}
E_0^\perp - \displaystyle \sum_{\s > \underline 0} b_{\s} \sum_{\alpha = 1}^{d_{\s}} \psi_{\alpha}^{\s} (\boldsymbol{M}) E_0^\perp \psi_{\alpha}^{\s} (\boldsymbol{M})^* \geq 0,    
\end{equation}
where the sequence $\{b_{\s}\}_{\s > \underline 0}$ of real numbers is given by Corollary \ref{k-inv-cnp}. To prove that $K$ is a CNP kernel, it suffices to show $b_{\s} \geq 0$ for each signature $\s > \underline 0$, thanks to Corollary \ref{k-inv-cnp}.

Let $N \geq 1$ be an arbitrary natural number, $\s'$ be a signature of length $N+1$ and $\psi_{\alpha'}^{\s'}$ be an arbitrary element of the orthonormal basis of $\mathcal{P}_{\s'}$. From \cite[Proposition 4.11.36]{upmeier}, it follows that 
\begin{equation}\label{main thm:eqn1}
\psi_{\alpha}^{\s} (\boldsymbol{M})^* \psi_{\alpha'}^{\s'} =
\begin{cases}
0, & \text{if $|\s'| < |\s|$}\\
\displaystyle \sum_{|\tilde{\s}| = |\s'| - |\s|} \frac{a_{\tilde{\s}'}}{a_{\s'}} \left(\psi_{\alpha}^{\s} (\partial) \psi_{\alpha'}^{\s'} \right)_{\tilde{\s}} , & \text{if $|\s'| \geq |\s|$}
\end{cases}
\end{equation}
where the summation is taken over all signatures $\tilde{\s}$ such that $|\tilde{\s}| = |\s'| - |\s|$, $\partial$ denotes the tuple of partial differentials $(\partial_{z_1}, \ldots, \partial_{z_d})$ and $\left(\psi_{\alpha}^{\s} (\partial) \psi_{\alpha'}^{\s'} \right)_{\tilde{\s}}$ denotes the component of $\psi_{\alpha}^{\s} (\partial) \psi_{\alpha'}^{\s'}$ in $\mathcal{P}_{\tilde{\s}}$ (see also \cite[Lemma 15]{Up}). 
Let $\boldsymbol{w}$ be an arbitrary element of $\Omega$. Therefore, from Equation \eqref{main thm:eqn1}, we have 
$$\displaystyle \sum_{\s > \underline 0} b_{\s} \sum_{\alpha = 1}^{d_{\s}} \psi_{\alpha}^{\s} (\boldsymbol{M}) E_0^\perp \psi_{\alpha}^{\s} (\boldsymbol{M})^* K_{\s'}(\cdot, \boldsymbol{w}) = \displaystyle \sum_{|\s| \leq N} b_{\s} \sum_{\alpha = 1}^{d_{\s}} \psi_{\alpha}^{\s} (\boldsymbol{M})  \psi_{\alpha}^{\s} (\boldsymbol{M})^* K_{\s'}(\cdot, \boldsymbol{w}).$$

From \cite[Page 5]{Up}, it follows that there exists a scalar $\gamma(\s')$ such that 
\begin{equation}\label{main thm:eqn2}
\displaystyle \sum_{|\s| \leq N} b_{\s} \sum_{\alpha = 1}^{d_{\s}} \psi_{\alpha}^{\s} (\boldsymbol{M}) \psi_{\alpha}^{\s} (\boldsymbol{M})^* K_{\s'}(\cdot, \boldsymbol{w}) = \gamma(\s') K_{\s'}(\cdot, \boldsymbol{w}).    
\end{equation}
To determine the value of $\gamma(\s')$, take inner product with $K(\cdot, \boldsymbol{w})$ to the both sides of Equation \eqref{main thm:eqn2} and observe that
\begin{flalign*}
\gamma(\s')K_{\s'}(\boldsymbol{w}, \boldsymbol{w}) &= \left\langle \displaystyle \sum_{|\s| \leq N} b_{\s} \sum_{\alpha = 1}^{d_{\s}} \psi_{\alpha}^{\s} (\boldsymbol{M}) \psi_{\alpha}^{\s} (\boldsymbol{M})^* K_{\s'}(\cdot, \boldsymbol{w}), K(\cdot, \boldsymbol{w}) \right\rangle\\
&= \sum_{|\s| \leq N} b_{\s} \left\langle K_{\s'}(\cdot, \boldsymbol{w}), \sum_{\alpha = 1}^{d_{\s}} \psi_{\alpha}^{\s}(\boldsymbol{M}) \psi_{\alpha}^{\s}(\boldsymbol{M})^* K(\cdot, \boldsymbol{w}) \right\rangle\\
&= \sum_{|\s| \leq N} b_{\s} \left\langle K_{\s'}(\cdot, \boldsymbol{w}), K_{\s}(\cdot, \boldsymbol{w}) K(\cdot, \boldsymbol{w}) \right\rangle\\
&= \sum_{|\s| \leq N} b_{\s} \left\langle K_{\s'}(\cdot, \boldsymbol{w}), K_{\s}(\cdot, \boldsymbol{w}) \sum_{\tilde{\s} \geq \underline 0} a_{\tilde{\s}} K_{\tilde{s}}(\cdot, \boldsymbol{w}) \right\rangle\\
\end{flalign*}
\begin{flalign*}
&= \sum_{|\s| \leq N} b_{\s} \left\langle K_{\s'}(\cdot, \boldsymbol{w}), \sum_{\tilde{\s} \geq \underline 0} a_{\tilde{\s}} \sum_{|\p| = |\s| + |\tilde{s}|} c^{\p}_{\s, \tilde{\s}} K_{\p}(\cdot, \boldsymbol{w}) \right\rangle\\
&= \left(\sum_{\substack{|\s| \leq N, \tilde{\s} \geq \underline 0,\\ |\s| + |\tilde{\s}| = |\s'|}} a_{\tilde{\s}} b_{\s} c^{\s'}_{\s, \tilde{\s}} \right)\frac{K_{\s'}(\boldsymbol{w}, \boldsymbol{w}) }{a_{\s'}}. 
\end{flalign*}
Here, the fifth equality holds because of Equation \eqref{exist:eqn1 b_s}. Therefore, we have 
\begin{equation}\label{main thm:eqn3}
\gamma(\s') =  \frac{1}{a_{\s'}} \left(\sum_{\substack{|\s| \leq N, \tilde{\s} \geq \underline 0,\\ |\s| + |\tilde{\s}| = |\s'|}} a_{\tilde{\s}} b_{\s} c^{\s'}_{\s, \tilde{\s}} \right). 
\end{equation}
Finally, evaluating the left hand side of Equation \eqref{main thm:eqn0} at $K_{\s'}(\cdot, \boldsymbol{w})$ and then, taking inner product with $K_{\s'}(\cdot, \boldsymbol{w})$, we obtain
\begin{flalign*}
0 &\leq \left\langle E_0^\perp K_{\s'}(\cdot, \boldsymbol{w}) - \displaystyle \sum_{\s > \underline 0} b_{\s} \sum_{\alpha = 1}^{d_{\s}} \psi_{\alpha}^{\s} (\boldsymbol{M}) E_0^\perp \psi_{\alpha}^{\s} (\boldsymbol{M})^* K_{\s'}(\cdot, \boldsymbol{w}), K_{\s'}(\cdot, \boldsymbol{w}) \right\rangle\\
&= \left\langle K_{\s'}(\cdot, \boldsymbol{w}) - \displaystyle \sum_{\s > \underline 0} b_{\s} \sum_{\alpha = 1}^{d_{\s}} \psi_{\alpha}^{\s} (\boldsymbol{M}) E_0^\perp \psi_{\alpha}^{\s} (\boldsymbol{M})^* K_{\s'}(\cdot, \boldsymbol{w}), K_{\s'}(\cdot, \boldsymbol{w}) \right\rangle\\
&= \|K_{\s'}(\cdot, \boldsymbol{w})\|^2 \left(1 - \gamma(\s')\right)\\
&= \frac{\|K_{\s'}(\cdot, \boldsymbol{w})\|^2 }{a_{\s'}}\left(a_{\s'} - \sum_{\substack{|\s| \leq N, \tilde{\s} \geq \underline 0,\\ |\s| + |\tilde{\s}| = |\s'|}} a_{\tilde{\s}} b_{\s} c^{\s'}_{\s, \tilde{\s}}\right)\\
&= \frac{\|K_{\s'}(\cdot, \boldsymbol{w})\|^2 }{a_{\s'}}\left(\sum_{\substack{\s, \tilde{\s} \geq \underline 0,\\ |\s| + |\tilde{\s}| = |\s'|}} a_{\tilde{\s}} b_{\s} c^{\s'}_{\s, \tilde{\s}} - \sum_{\substack{|\s| \leq N, \tilde{\s} \geq \underline 0,\\ |\s| + |\tilde{\s}| = |\s'|}} a_{\tilde{\s}} b_{\s} c^{\s'}_{\s, \tilde{\s}}\right)\\
&= \frac{\|K_{\s'}(\cdot, \boldsymbol{w})\|^2 }{a_{\s'}} b_{\s'}c^{\s'}_{\s', 0}.
\end{flalign*}
This implies that $b_{\s'} \geq 0$. Note that, in the above series of equalities, the third equality holds because of Equation \eqref{main thm:eqn3} and the fourth equality holds due to Equation \eqref{cha:eqn:1}. This proves that $K$ is a CNP kernel.
\end{proof}

% As a consequence of the theorem above, we observe that the tuple of multiplication operators by the coordinate functions on $\mathcal H^{(\nu)}(\Omega)$ does not possess characteristic function. This is recorded is the following corollary.

% \begin{corollary}
% Let $\Omega$ be a Cartan domain of rank $r > 1$. For any $\nu \in \mathcal{W}_{\Omega}$, suppose $\boldsymbol{M}^{(\nu)} = (M_{z_1}, \ldots, M_{z_d})$ denotes the tuple of multiplication operators by the coordinate functions on the weighted Bergman space $\mathbb A^{(\nu)}(\Omega)$. Then, the operator tuple $\boldsymbol{M}^{(\nu)}$ does not admit a characteristic function.    
% \end{corollary}

% \begin{proof}
% The proof immediately follows from Proposition \ref{BermanNotCNP} and Theorem \ref{cnp main}.   
% \end{proof}

\begin{remark}
Let $\Omega$ be a Cartan domain of rank $r > 1$ and for $\nu \in \mathcal{W}_{\Omega}$, let $K^{(\nu)}$ be the weighted Bergman kernel on $\Omega$. A direct consequence of Proposition \ref{BermanNotCNP} and Theorem \ref{cnp main} imply that there exists a $\frac{1}{K^{(\nu)}}$-contraction $\boldsymbol{T} = (T_1, \ldots, T_d)$ such that $\boldsymbol{T}$ does not admit a characteristic function. 

For $\nu > \frac{d}{r}$, the tuple of multiplication operators $\boldsymbol M^{(\nu)}$ on the weighted Bergman space $\mathbb A^{(\nu)}(\Omega)$ is a $\frac{1}{K^{(\nu)}}$-contraction, \cite[Theorem 3.3]{Arazy2}. Lemma \ref{lem2.1} yields that $\boldsymbol M^{(\nu)}$ is pure and $\Delta_{\boldsymbol{M}^{(\nu)}} = E_0$.  Recall that $\{\psi_{\s}^{\alpha} : \s \geq \underline 0, 1 \leq \alpha \leq d_{\s}\}$ is a complete orthogonal set of $\mathbb A^{(\nu)}(\Omega)$. For any $N \geq 0$, let $\mathcal{M}_N$ be the closed subspace of $\mathbb A^{(\nu)}(\Omega)$ spanned by the set $\{\psi_{\s}^{\alpha} : |\s| \leq N, 1 \leq \alpha \leq d_{\s}\}$ and $\boldsymbol T_{N} := P_{N} \boldsymbol M^{(\nu)}|_{\mathcal{M}_N}$, where $P_{N}$ is the orthogonal projection of $\mathbb A^{(\nu)}(\Omega)$ onto $\mathcal{M}_N$. It is easy to verify that $\boldsymbol{T}_N$ is a pure $\frac{1}{K^{(\nu)}}$-contraction and $\Delta_{\boldsymbol{T}_N} = E_{0}$. Furthermore, an adaption of the proof of \cite[Proposition 3.3]{Tiret} establishes that $\boldsymbol{T}_N$ does not admit any characteristic function for $\nu \geq \frac{d}{r} + 1$.
\end{remark}

\section{Characteristic Function: Construction}

In this section, we explicitly construct the characteristic function of a $d$-tuple of $\frac{1}{K}$-contraction for certain $\mathbb K$-invariant kernels $K$. Let $K = \sum_{\s \geq \underline 0} a_{\s}K_{\s}$ be a $\mathbb K$-invariant kernel on $\Omega$ and $\boldsymbol{T} = (T_1, \ldots, T_d)$ be a $\frac{1}{K}$-contraction. Then, recall that, the series
$$\displaystyle \sum_{\s > \underline 0} b_{\s} \sum_{\alpha = 1}^{d_{\s}} \psi_\alpha^{\s}(\boldsymbol{T}) \psi_\alpha^{\s}(\boldsymbol{T})^*$$
is convergent in strong operator topology and 
$$\displaystyle \sum_{\s > \underline 0} b_{\s} \sum_{\alpha = 1}^{d_{\s}} \psi_\alpha^{\s}(\boldsymbol{T}) \psi_\alpha^{\s}(\boldsymbol{T})^* \leq I.$$
%where $\{b_{\s}\}_{\s>\underline 0}$ is the sequence of real numbers, obtained in Proposition \ref{existence of b_s}, such that $1 - \frac{1}{K} = \sum_{\s >\underline 0} b_{\s} K_{\s}$. 
The positive square root of $I - \displaystyle \sum_{\s > \underline 0} b_{\s} \sum_{\alpha = 1}^{d_{\s}} \psi_\alpha^{\s}(\boldsymbol{T}) \psi_\alpha^{\s}(\boldsymbol{T})^*$ is denoted by $\Delta_{\boldsymbol{T}}$. Also, recall that a $\frac{1}{K}$-contraction $\boldsymbol{T} = (T_1, \ldots, T_d)$ is said to admit a characteristic function if there exists a Hilbert space $\mathcal{E}$ and an analytic function $\Theta_{\boldsymbol{T}} : \Omega \to \mathcal{B}(\mathcal{E}, \overline{\mbox{Ran}} \Delta_{\boldsymbol{T}})$ such that the corresponding multiplication operator $M_{\Theta_{\boldsymbol{T}}} : H_{K} \otimes \mathcal{E} \to H_{K} \otimes \overline{\mbox{Ran}}\,\, \Delta_{\boldsymbol{T}}$ satisfies
$I - V_{\boldsymbol{T}} V_{\boldsymbol{T}}^* = M_{\Theta_{\boldsymbol{T}}} M_{\Theta_{\boldsymbol{T}}}^*.$

% \begin{definition}\label{def:existence of cha fun}
% Let $\boldsymbol{T} = (T_1, \ldots, T_d)$ be a $1/K-$contraction. The tuple $\boldsymbol{T}$ is said to admit a characteristic function if there exists a Hilbert space $\mathcal{E}$ and an analytic function $\theta_{\boldsymbol{T}} : \Omega \to \mathcal{B}(\mathcal{E}, \overline{\mbox{Ran}} \Delta_{\boldsymbol{T}})$ such that the corresponding multiplication operator $M_{\Theta_{\boldsymbol{T}}} : H_{K} \otimes \mathcal{E} \to H_{K} \otimes \overline{\mbox{Ran}} \Delta_{\boldsymbol{T}}$ fulfilling the following identity
% $$I - V_{\boldsymbol{T}} V_{\boldsymbol{T}}^* = M_{\Theta_{\boldsymbol{T}}} M_{\Theta_{\boldsymbol{T}}}^*.$$
% \end{definition}

\subsection{Functional Calculus}
To explicitly compute the characteristic function of a $\frac{1}{K}$-contraction $\boldsymbol{T} = (T_1, \ldots, T_d)$ on a Hilbert space $\mathcal{H}$, we need the existence of the functional calculus $K(\boldsymbol{T}, \boldsymbol{w})$ for each $\boldsymbol{w} \in \Omega$. In order to define the operator $K(\boldsymbol{T}, \boldsymbol{w})$ for each $\boldsymbol{w} \in \Omega$, we assume the followings.
\begin{itemize}
    \item[(A)] For each $\boldsymbol{w} \in \Omega$, $\sum_{\s} a_{\s} K_{\s}(\boldsymbol{M}, \boldsymbol{w})$ and $\sum_{\s > \underline 0} b_{\s}K_{\s}(\boldsymbol{M}, \boldsymbol{w})$ converges in the strong operator topology.
    \item[(B)] For some $c > 0$, $\sum_{\s \geq \underline 0} a_{\s} \sum_{\alpha}^{d_{\s}} \psi_{\alpha}^{\s}(\boldsymbol{T}) \Delta_{\boldsymbol{T}}^2 \psi_{\alpha}^{\s}(\boldsymbol{T})^* \geq c I$.
\end{itemize}
For the rest of this section, we assume Conditions (A) and (B). \textit{Note that if $\boldsymbol{T}$ is a pure $\frac{1}{K}$-contraction, then $\boldsymbol{T}$ satisfies Condition (B)}.

We observe that if $H_K$ is the subspace consisting of holomorphic functions in some $L^2$ space, then Condition (A) automatically holds.

\begin{lemma} \label{pp4}
If $H_K$ is the subspace consisting of holomorphic functions in some $L^2$ space (with the inherited norm), then any bounded holomorphic function $f$ on $\Omega$ is a bounded multiplier of $H_K$, and the corresponding multiplication operator $M_f$ satisfies $\|M_f\|\leq\|f\|_\infty$.
\end{lemma}

\begin{proof} Immediate from $\|M_fu\|^2 = \int |fu|^2 \leq \|f\|_\infty^2 \int |u|^2 = \|f\|_\infty^2 \|u\|^2$.

\end{proof}

\begin{corollary} \label{pp5}
Under the hypothesis of Lemma \ref{pp4}, for any $F$ as in Proposition \ref{pp2} and $\boldsymbol w\in\Omega$, $F(\cdot,\boldsymbol w)$ is a bounded multiplier on $H_K$. In particular, $K(\cdot,\boldsymbol w)$ is a bounded multiplier, and if $K$ is zero-free, then so is $1/K(\cdot,\boldsymbol w)$.

\end{corollary}

\begin{proof} Immediate from Corollary \ref{pp3}. \end{proof}

\begin{corollary} \label{pp6}
Again under the hypothesis of Lemma \ref{pp4}, for any $F$ as in Proposition \ref{pp2} and $\boldsymbol w\in\Omega$, the series
$$ \sum_{\s} c_{\s} K_{\s}(\boldsymbol{M}, \boldsymbol w) ,  $$
where $\boldsymbol{M}$ denotes the commuting tuple of multiplications by the coordinate functions, converges to the operator of multiplication by $F(\cdot,\boldsymbol w)$ on $\mathcal H_K$ in operator norm. In particular, this holds for $F=K$, and if $K$ is zero-free, then also for $F=1/K$.
\end{corollary}

\begin{proof} By the lemma \ref{pp4}, $\|f_m-f\|_\infty\to0$ implies $\|M_{f_m}-M_f\|\to0$. Take $f_m:=\sum_{|\s|\leq m} c_{\s} K_{\s}(\cdot,\boldsymbol w)$ and use Proposition \ref{pp2}.
\end{proof}

The corollary above proves that if the norm of the reproducing kernel Hilbert space $H_K$ is given by an integral, then Condition (A) is always satisfied. Now, we provide a few necessary lemmas. 

\begin{lemma}\label{V_T injective}
The operator $V_{\boldsymbol{T}} : \mathcal{H} \to H_{K} \otimes \overline{\mbox{Ran}}\,\, \Delta_{\boldsymbol{T}}$, given by Equation \eqref{char:eqn:def V_T}, is bounded below.   
\end{lemma}

\begin{proof}
Let $h$ be an element in $\mathcal{H}$. Then, we have
\begin{flalign*}
\left\langle V_{\boldsymbol{T}} h, V_{\boldsymbol{T}} h \right\rangle
= \left\langle \displaystyle \sum_{\s \geq \underline 0} a_{\s} \sum_{\alpha = 1}^{d_{\s}} \psi_{\alpha}^{\s}(\boldsymbol{T}) \Delta_{\boldsymbol{T}}^2 \psi_{\alpha}^{\s}(\boldsymbol{T})^* h , h \right\rangle
 \geq c \langle h, h\rangle.
\end{flalign*}
This shows that $\|V_{\boldsymbol{T}} h\| \geq \sqrt{c}\|h\|$ holds for every $h \in \mathcal{H}$. Therefore, $V_{\boldsymbol{T}}$ is bounded below.
\end{proof}

\begin{remark}\label{new:rem}
In particular, $V_{\boldsymbol{T}}$ has closed range, hence (by Banach's Closed Range Theorem) so has $V^*_{\boldsymbol{T}}$; thus Ran $V^*_{\boldsymbol{T}}$ is all of $\mathcal H$.    
\end{remark}

\begin{lemma}\label{lem:fun cal}
For each $\boldsymbol{w} \in \Omega$, both the series 
$$\displaystyle \sum_{\s \geq \underline 0} a_{\s} \sum_{\alpha = 0}^{d_{\s}} \overline{\psi_{\alpha}^{\s}(\boldsymbol{w})} \psi_{\alpha}^{\s}(\boldsymbol{T})\,\,\mbox{and}\,\,\displaystyle \sum_{\s > \underline 0} b_{\s} \sum_{\alpha = 0}^{d_{\s}} \overline{\psi_{\alpha}^{\s}(\boldsymbol{w})} \psi_{\alpha}^{\s}(\boldsymbol{T})$$
converge in strong operator topology.
\end{lemma}

\begin{proof}
Condition (A) implies that the series $\sum_{\s \geq \underline 0} a_{\s} \sum_{\alpha = 0}^{d_{\s}} \overline{\psi_{\alpha}^{\s}(\boldsymbol{w})} \psi_{\alpha}^{\s}(\boldsymbol{M}) \otimes I_{\overline{\mbox{Ran}} \Delta_{\boldsymbol{T}}}$ converges in strong operator topology. Let $\epsilon > 0$ and $f \in H_{K} \otimes \overline{\mbox{Ran}}\,\, \Delta_{\boldsymbol{T}}$. Then there exists $N \geq 1$ such that 
$$\| \displaystyle \sum_{|\s| = N_1}^{N_2} a_{\s} \sum_{\alpha = 0}^{d_{\s}} \overline{\psi_{\alpha}^{\s}(\boldsymbol{w})} \psi_{\alpha}^{\s}(\boldsymbol{M}) \otimes I_{\overline{\mbox{Ran}} \,\,\Delta_{\boldsymbol{T}}} f\| < \epsilon$$
holds for every $N_1, N_2 \geq N$. This implies that 
\begin{flalign*}
&\| \displaystyle \sum_{|\s| = N_1}^{N_2} a_{\s} \sum_{\alpha = 0}^{d_{\s}} \overline{\psi_{\alpha}^{\s}(\boldsymbol{w})} \psi_{\alpha}^{\s}(\boldsymbol{T}) V_{\boldsymbol{T}}^* f\| \\
&= \| V_{\boldsymbol{T}}^* \displaystyle \sum_{|\s| = N_1}^{N_2} a_{\s} \sum_{\alpha = 0}^{d_{\s}} \overline{\psi_{\alpha}^{\s}(\boldsymbol{w})} \psi_{\alpha}^{\s}(\boldsymbol{M}) \otimes I_{\overline{\mbox{Ran}}\,\, \Delta_{\boldsymbol{T}}} f\|\\
&\leq \| \displaystyle \sum_{|\s| = N_1}^{N_2} a_{\s} \sum_{\alpha = 0}^{d_{\s}} \overline{\psi_{\alpha}^{\s}(\boldsymbol{w})} \varphi_{\alpha}^{\s}(\boldsymbol{M}) \otimes I_{\overline{\mbox{Ran}} \,\,\Delta_{\boldsymbol{T}}} f\| < \epsilon,\,\,N_1, N_2 \geq N.
\end{flalign*}
Here, the first inequality occurs since $V_{\boldsymbol{T}}$ is a contraction. This proves that the series $ \sum_{\s \geq \underline 0} a_{\s} \sum_{\alpha = 0}^{d_{\s}} \overline{\psi_{\alpha}^{\s}(\boldsymbol{w})} \psi_{\alpha}^{\s}(\boldsymbol{T})$ converges in strong operator topology on $\mbox{Ran} V_{\boldsymbol{T}}^*$. Remark \ref{new:rem} yields that $\mbox{Ran}\,\,V_{\boldsymbol{T}}^*$ is $\mathcal{H}$. This proves that $ \sum_{\s \geq \underline 0} a_{\s} \sum_{\alpha = 0}^{d_{\s}} \overline{\psi_{\alpha}^{\s}(\boldsymbol{w})} \psi_{\alpha}^{\s}(\boldsymbol{T})$ converges in strong operator topology on $\mathcal{H}$.

A similar argument shows that the series $\sum_{\s > \underline 0} b_{\s} \sum_{\alpha = 0}^{d_{\s}} \overline{\psi_{\alpha}^{\s}(\boldsymbol{w})} \psi_{\alpha}^{\s}(\boldsymbol{T})$ converges in the strong operator topology.
\end{proof}

Recall that if $K$ is a $\mathbb K$-invariant kernel, then there exists a sequence $\{b_{\s}\}$ of real numbers such that $1 - \frac{1}{K} = \sum_{\s>\underline 0} b_{\s} K_{\s}$ holds. This implies that
\begin{equation}\label{cha:fun:cal:inv}
\left( \sum_{\s \geq \underline 0} a_{\s} \sum_{\alpha = 0}^{d_{\s}} \overline{\psi_{\alpha}^{\s}(\boldsymbol{w})} \psi_{\alpha}^{\s}(\boldsymbol{T}) \right)  \left(I - \sum_{\s > \underline 0} b_{\s} \sum_{\alpha = 0}^{d_{\s}} \overline{\psi_{\alpha}^{\s}(\boldsymbol{w})} \psi_{\alpha}^{\s}(\boldsymbol{T}) \right)  = I.    
\end{equation}
For any $\boldsymbol{w} \in \Omega$, let 
$$K_{\boldsymbol{w}} (\boldsymbol{T}) = \displaystyle \sum_{\s \geq \underline 0} a_{\s} \sum_{\alpha}^{d_{\s}} \overline{\psi_{\alpha}^{\s}(\boldsymbol{w})} \psi_{\alpha}^{\s} (\boldsymbol{T}).$$
Then, for any $\boldsymbol{z} \in \Omega,$ we have 
$$K_{\boldsymbol{z}} (\boldsymbol{T})^* = \displaystyle \sum_{\s \geq \underline 0} a_{\s} \sum_{\alpha}^{d_{\s}} \psi_{\alpha}^{\s}(\boldsymbol{z}) \psi_{\alpha}^{\s} (\boldsymbol{T})^*.$$
Equation \eqref{cha:fun:cal:inv} implies that
$$K_{\boldsymbol{z}}(\boldsymbol{T})^* = \left(I - \sum_{\s > \underline 0} b_{\s} \sum_{\alpha = 0}^{d_{\s}} \psi_{\alpha}^{\s}(\boldsymbol{z}) \psi_{\alpha}^{\s}(\boldsymbol{T})^* \right)^{-1}.$$

\subsection{Characteristic Function}
In this subsection, we give an explicit construction of the characteristic function of $\boldsymbol{T}$ with the assumption that $K$ is a CNP kernel. The CNP property of $K$ implies that each $b_{\s}$ must be a non-negative real number. Let $\tilde{\mathcal{H}} = \bigoplus_{\s > \underline 0} \bigoplus_{\alpha}^{d_{\s}} \mathcal{H}$. Note that $\tilde{\mathcal{H}}$ is a direct sum of countably many copies of the Hilbert space $\mathcal{H}$. Every $\boldsymbol{z} \in \Omega$ gives rise to an operator 
\begin{equation}\label{cha:eqn:op Z}
    \boldsymbol{Z} = \left(\sqrt{b_{\s}}\psi_\alpha^{d_{\s}}(\boldsymbol{z}) I\right)_{\s>\underline 0, \alpha=1}^{d_{\s}} : \tilde{\mathcal{H}} \to \mathcal{H},
\end{equation}
mapping $\left(h_{\alpha}^{\s}\right)_{\s>\underline 0, \alpha=1}^{d_{\s}}$ to $\sum_{\s>\underline 0} \sqrt{b_{\s}} \sum_{\alpha=1}^{d_{\s}} \psi_{\alpha}^{\s}(\boldsymbol{z}) h_{\alpha}^{\s}$ and
\begin{flalign*}
\|\boldsymbol{Z}\left(h_{\alpha}^{\s}\right)_{\s>\underline 0, \alpha=1}^{d_{\s}}\|^2 &= \|\displaystyle \sum_{\s>\underline 0} \sqrt{b_{\s}} \sum_{\alpha=1}^{d_{\s}} \psi_{\alpha}^{\s}(\boldsymbol{z}) h_{\alpha}^{\s}\|^2\\
&\leq \left(\displaystyle \sum_{\s>\underline 0} \sum_{\alpha = 1}^{d_{\s}} b_{\s} |\psi_{\alpha}^{\s}(\boldsymbol{z})|^2\right) \left(\displaystyle \sum_{\s>0} \sum_{\alpha = 1}^{d_{\s}} \|h_{\alpha}^{\s}\|^2\right)\\
&= \left(1 - \frac{1}{K(\boldsymbol{z}, \boldsymbol{z})}\right)\left(\displaystyle \sum_{\s>\underline 0} \sum_{\alpha = 1}^{d_{\s}} \|h_{\alpha}^{\s}\|^2\right). 
\end{flalign*}
This proves that $\| \boldsymbol{Z}\| \leq \left(1 - \frac{1}{K(\boldsymbol{z}, \boldsymbol{z})} \right)^{1/2} < 1$. Consider the operator 
$$\tilde{\boldsymbol{T}} := \left(\sqrt{b_{\s}} \psi_{\alpha}^{\s}(\boldsymbol{T})\right)_{\s > \underline 0, \alpha = 1}^{d_{\s}} : \tilde{\mathcal{H}} \to \mathcal{H},$$ 
defined by,
$$\tilde{\boldsymbol{T}} \left(\left(h_{\alpha}^{\s}\right)_{\s>\underline 0, \alpha=1}^{d_{\s}}\right) = \displaystyle \sum_{\s > \underline 0} \sum_{\alpha = 1}^{d_{\s}} \sqrt{b_{\s}} \psi_{\alpha}^{\s}(\boldsymbol{T}) h_{\alpha}^{\s}.$$
Then, it is easily verified that $\tilde{\boldsymbol{T}}^* : \mathcal{H} \to \tilde{\mathcal{H}}$ is given by $\tilde{\boldsymbol{T}}^*h = \left( \sqrt{b_{\s}} \psi_{\alpha}^{\s}(\boldsymbol{T})^*h \right)_{\s>0, \alpha = 1}^{d_{\s}}$, $h \in \mathcal{H}$. Using the fact that $\boldsymbol{T}$ is a $\frac{1}{K}$-contraction, it is straight forward to verify that $\tilde{\boldsymbol{T}}^*$ is a contraction. Also, we have $\Delta_{\boldsymbol{T}}^2 = I_{\mathcal{H}} - \tilde{\boldsymbol{T}} \tilde{\boldsymbol{T}}^*.$ Let $D_{\tilde{\boldsymbol{T}}} := \left(I_{\tilde{\mathcal{H}}} - \tilde{\boldsymbol{T}}^* \tilde{\boldsymbol{T}} \right)^\frac{1}{2}$ and $\mathcal{D}_{\tilde{\boldsymbol{T}}} := \overline{\mbox{Ran}} D_{\tilde{\boldsymbol{T}}}$. The identity
\begin{equation}\label{cha:eqn:TDT}
\tilde{\boldsymbol{T}} D_{\tilde{\boldsymbol{T}}} = \Delta_{\boldsymbol{T}} \tilde{\boldsymbol{T}}
\end{equation}
follows from a direct computation. For each $\boldsymbol{z} \in \Omega$, since the operator $\boldsymbol{Z}$, defined by Equation \eqref{cha:eqn:op Z}, is a strict contraction, the operator $I_{\mathcal{H}} - \boldsymbol{Z}\tilde{\boldsymbol{T}}^*$ is invertible. Now, we are prepared to provide the definition of the characteristic function of $\boldsymbol{T}$.  

\begin{definition}\label{def:cha fun}
The characteristic function $\theta_{\boldsymbol{T}} : \Omega \to \mathcal{B}(\mathcal{D}_{\tilde{\boldsymbol{T}}}, \overline{\mbox{Ran}}\,\,\Delta_{\boldsymbol{T}})$ of $\boldsymbol{T}$ is defined by
\begin{equation}\label{cha:eqn:cha}
 \theta_{\boldsymbol{T}}(\boldsymbol{z}) := \left(-\tilde{\boldsymbol{T}} + \Delta_{\boldsymbol{T}}(I_{\mathcal{H}} - \boldsymbol{Z}\tilde{\boldsymbol{T}}^*)^{-1}\boldsymbol{Z}D_{\tilde{\boldsymbol{T}}}\right)|_{\mathcal{D}_{\tilde{\boldsymbol{T}}}},\,\,\boldsymbol{z} \in \Omega,   
\end{equation}
where, for each $\boldsymbol{z}$ in $\Omega$, the operator $\boldsymbol{Z}$ is given in Equation \eqref{cha:eqn:op Z}.
\end{definition}
The relation $\tilde{\boldsymbol{T}} D_{\tilde{\boldsymbol{T}}} = \Delta_{\boldsymbol{T}} \tilde{\boldsymbol{T}}$ implies that $\theta_{\boldsymbol{T}}(\boldsymbol{z}) \in \mathcal{B}(\mathcal{D}_{\tilde{\boldsymbol{T}}}, \overline{\mbox{Ran}} \Delta_{\boldsymbol{T}})$ for each $\boldsymbol{z} \in \Omega.$ Therefore, the characteristic function $\theta_{\boldsymbol{T}}$ of $\boldsymbol{T}$ is a $\mathcal{B}(\mathcal{D}_{\tilde{\boldsymbol{T}}}, \overline{\mbox{Ran}} \Delta_{\boldsymbol{T}})$-valued analytic function on $\Omega$. Also, note that if $\boldsymbol{z} \in \Omega$, then
$$\boldsymbol{Z}\tilde{\boldsymbol{T}}^* = \displaystyle \sum_{\s>\underline 0}b_{\s} \sum_{\alpha = 1}^{d_{\s}} \psi_{\alpha}^{\s}(\boldsymbol{z}) \psi_{\alpha}^{\s}(\boldsymbol{T})^*$$
and consequently, we have
$$\left(I_{\mathcal{H}} - \boldsymbol{Z}\tilde{\boldsymbol{T}}^*\right)^{-1} = K_{\boldsymbol{z}}(\boldsymbol{T})^*.$$
The proofs of the following lemmas are consequences of a straightforward computation and therefore omitted. 

\begin{lemma}\label{lem:cha iden}
For every $\boldsymbol{z}, \boldsymbol{w} \in \Omega$, we have the following identity
$$I - \theta_{\boldsymbol{T}}(\boldsymbol{z}) \theta_{\boldsymbol{T}}(\boldsymbol{w})^* = \frac{1}{K(\boldsymbol{z}, \boldsymbol{w})} \Delta_{\boldsymbol{T}} \left(K_{\boldsymbol{z}}(\boldsymbol{T})\right)^* K_{\boldsymbol{w}}(\boldsymbol{T}) \Delta_{\boldsymbol{T}}.$$
\end{lemma}

\begin{lemma}\label{lem:V_T*}
For any $\boldsymbol{w} \in \Omega$ and $\xi \in \overline{\mbox{Ran}}\,\Delta_{\boldsymbol{T}}$, we have
$$V_{\boldsymbol{T}}^* \left(K_{\boldsymbol{w}} \otimes \xi\right) = K_{\boldsymbol{w}} (\boldsymbol{T}) \Delta_{\boldsymbol{T}} \xi,$$
where the operator $V_{\boldsymbol{T}}$ is given by Equation \eqref{char:eqn:def V_T}.
\end{lemma}

The proof of the corollary below follows immediately from Lemma \ref{lem:cha iden} and Lemma \ref{lem:V_T*}.

\begin{corollary}\label{cor:V_T^* inn}
For any $\boldsymbol{z}, \boldsymbol{w} \in \Omega$ and $\xi, \eta \in \overline{\mbox{Ran}} \,\Delta_{\boldsymbol{T}}$, we have
$$\left\langle V_{\boldsymbol{T}}^* (K_{\boldsymbol{w}} \otimes \xi), V_{\boldsymbol{T}}^* (K_{\boldsymbol{z}} \otimes \eta \right\rangle = K(\boldsymbol{z}, \boldsymbol{w}) \left\langle \left( I - \theta_{\boldsymbol{T}}(\boldsymbol{z}) \theta_{\boldsymbol{T}}(\boldsymbol{w})^* \right) \xi, \eta \right\rangle.$$
\end{corollary}

The characteristic function $\theta_{\boldsymbol{T}}$ of $\boldsymbol{T}$ gives rise to the multiplication operator $M_{\theta_{\boldsymbol{T}}} : H_{K} \otimes \mathcal{D}_{\tilde{\boldsymbol{T}}} \to H_{K} \otimes \overline{\mbox{Ran}}\,\Delta_{\boldsymbol{T}}$. The following theorem shows that the characteristic function $\theta_{\boldsymbol{T}}$ of $\boldsymbol{T}$, defined via Definition \ref{def:cha fun} satisfies the identity of Definition \ref{def:existence of cha fun}.

\begin{theorem}\label{thm:cha existence}
The multiplication operator $M_{\theta_{\boldsymbol{T}}} : H_{K} \otimes \mathcal{D}_{\tilde{\boldsymbol{T}}} \to H_{K} \otimes \overline{\mbox{Ran}}\,\Delta_{\boldsymbol{T}}$ corresponding to the characteristic function $\theta_{\boldsymbol{T}}$ of $\boldsymbol{T}$ satisfies the following identity
$$I - V_{\boldsymbol{T}} V_{\boldsymbol{T}}^* = M_{\theta_{\boldsymbol{T}}}M_{\theta_{\boldsymbol{T}}}^*.$$
\end{theorem}

\begin{proof}
The proof is similar to the proof of \cite[Theorem 4.11]{Tiret}. Therefore, the proof of the theorem is omitted.    
\end{proof}

As a direct consequence of the theorem above, we obtain the following corollary.

\begin{corollary}\label{cor:cha contrac}
The characteristic function $\theta_{\boldsymbol{T}}$ of $\boldsymbol{T}$ is a bounded analytic function on $\Omega$. In particular, $\sup_{\boldsymbol{z} \in \Omega} \|\theta_{\boldsymbol{T}}(\boldsymbol{z})\| \leq 1.$    
\end{corollary}

\begin{definition}\label{def:conincidence}
Let $\boldsymbol{T}$ and $\boldsymbol{R}$ be two $\frac{1}{K}$-contractions defined on Hilbert spaces $\mathcal{H}$ and $\mathcal{K}$, respectively. The characteristic functions $\theta_{\boldsymbol{T}}$ and $\theta_{\boldsymbol{R}}$ of $\boldsymbol{T}$ and $\boldsymbol{R}$, respectively, are said to coincide if there exist unitaries $\tau : \mathcal{D}_{\tilde{\boldsymbol{T}}} \to \mathcal{D}_{\tilde{\boldsymbol{R}}}$ and $\tau_* : \overline{\mbox{Ran}}\,\Delta_{\boldsymbol{T}} \to \overline{\mbox{Ran}}\,\Delta_{\boldsymbol{R}}$ such that 
$$\tau_* \theta_{\boldsymbol{T}}(\boldsymbol{z})  =  \theta_{\boldsymbol{R}}(\boldsymbol{z}) \tau$$
holds for every $\boldsymbol{z} \in \Omega$.
\end{definition}

\begin{theorem}\label{thm:model of pure}
If $\boldsymbol{T}$ is a pure $\frac{1}{K}$-contraction, then the tuple $\boldsymbol{T}$ is unitarily equivalent to the tuple $\left(P_{\mathbb H_{\boldsymbol{T}}} \left(M_{z_1} \otimes I_{\overline{\mbox{\tiny{Ran}}}\,\Delta_{\boldsymbol{T}}}\right)|_{\mathbb H_{\boldsymbol{T}}}, \ldots, P_{\mathbb H_{\boldsymbol{T}}} \left(M_{z_d} \otimes I_{\overline{\mbox{\tiny{Ran}}}\,\Delta_{\boldsymbol{T}}}\right)|_{\mathbb H_{\boldsymbol{T}}} \right)$, where
$\mathbb H_{\boldsymbol{T}} := \left(H_K \otimes \overline{\mbox{Ran}}\,\Delta_{\boldsymbol{T}} \right) \ominus M_{\theta_{\boldsymbol{T}}} \left(H_K \otimes \mathcal{D}_{\tilde{\boldsymbol{T}}}\right)$ and $P_{\mathbb H_{\boldsymbol{T}}}$ denote the orthogonal projection from $H_K \otimes \overline{\mbox{Ran}}\,\Delta_{\boldsymbol{T}}$ onto $\mathbb H_{\boldsymbol{T}}$.
\end{theorem}

\begin{proof}
The proof if similar to the proof of \cite[Theorem 3.7]{Tiret2} and therefore, omitted. \end{proof}

The main result of this section is the following theorem, which states that for certain $\mathbb K$-invariant kernels $K$, the characteristic function of a pure $\frac{1}{K}$-contraction $\boldsymbol{T}$ determines the unitary equivalence class of $\boldsymbol{T}$. The proof follows along the line of the proof of \cite[Theorem 4.4]{Tiret2}. Therefore, we leave out the proof.

\begin{theorem}
Let $K$ be a $\mathbb K$-invariant kernel satisfying Condition (A).
    Two pure $\frac{1}{K}$-contractions are unitarily equivalent if and only if their characteristic functions coincide.
\end{theorem}

\bibliographystyle{amsplain}

\end{document}